\documentclass[12pt, reqno]{amsart}

\usepackage{comment}
\usepackage{lscape}
\usepackage{graphicx}
\usepackage{amssymb,amsmath,amsthm,amscd}
\usepackage{mathrsfs}
\usepackage{enumerate}
\usepackage[usenames,dvipsnames]{color}
\usepackage[colorlinks=true, pdfstartview=FitV,
linkcolor=blue,citecolor=blue,urlcolor=blue]{hyperref}

\usepackage[all]{xy}
\usepackage{stmaryrd}
\usepackage{MnSymbol}

\allowdisplaybreaks[4]
\usepackage{oldgerm}

\usepackage[normalem]{ulem}  

\newcommand{\nc}{\newcommand}
\numberwithin{equation}{section}

\newenvironment{red}{\relax\color{red}}{\relax}
\newenvironment{blue}{\relax\color{blue}}{\hspace*{.5ex}\relax}
\newenvironment{jaune}{\relax\color{Orchid}}{\hspace*{.5ex}\relax}

\newcommand{\beb}{\begin{blue}}
	\newcommand{\eb}{\end{blue}}

\newcommand{\bj}{\begin{jaune}}
	\newcommand{\ej}{\end{jaune}}

\newcommand{\ber}{\begin{red}}
	\newcommand{\er}{\end{red}}

\theoremstyle{plain}
\newtheorem{lemma}{Lemma}[section]
\newtheorem{prop}[lemma]{Proposition}
\newtheorem{theorem}[lemma]{Theorem}

\newcommand{\Prop}{\begin{prop}}
	\newcommand{\enprop}{\end{prop}}
\newcommand{\Lemma}{\begin{lemma}}
	\newcommand{\enlemma}{\end{lemma}}
\newcommand{\Th}{\begin{theorem}}
	\newcommand{\enth}{\end{theorem}}
\newtheorem{corollary}[lemma]{Corollary}
\newcommand{\Cor}{\begin{corollary}}
	\newcommand{\encor}{\end{corollary}}
\newtheorem{definition}[lemma]{Definition}
\newtheorem*{conjecture}{Conjecture}
\newcommand{\Def}{\begin{definition}}
	\newcommand{\edf}{\end{definition}}
\newtheorem{sublemma}[lemma]{Sublemma}
\newcommand{\Sublemma}{\begin{sublemma}}
	\newcommand{\ensub}{\end{sublemma}}

\theoremstyle{definition}
\newtheorem{remark}[lemma]{Remark}
\newtheorem{example}[lemma]{Example}
\newtheorem{Convention}[lemma]{Convention}
\newcommand{\Conv}{\begin{Convention}}
	\newcommand{\enconv}{\end{Convention}}
\nc{\Con}{\begin{conjecture}}
	\nc{\encon}{\end{conjecture}}
\nc{\Rem}{\begin{remark}}
	\nc{\enrem}{\end{remark}}
\newcommand{\C}{{\mathbb C}}
\newcommand{\Q}{\mathbb {Q}}

\newcommand{\Z}{{\mathbb Z}}

\newcommand{\B}{{\mathbf{B}}}

\def\inv{^{-1}}
\newcommand{\g}{{\mathfrak{g}}}

\newcommand{\End}{\operatorname{End}}

\newcommand{\isoto}[1][]{\mathop{\xrightarrow%
		[{\raisebox{.3ex}[0ex][.3ex]{$\scriptstyle{#1}$}}]%
		{{\raisebox{-.6ex}[0ex][-.6ex]{$\mspace{2mu}\sim\mspace{2mu}$}}}}}

\newcommand{\eq}{\begin{eqnarray}}
	\newcommand{\eneq}{\end{eqnarray}}

\newcommand{\eqa}{\begin{align}}
	\newcommand{\eneqa}{\end{align}}

\newcommand{\eqn}{\begin{eqnarray*}}
	\newcommand{\eneqn}{\end{eqnarray*}}
\newcommand{\on}{\operatorname}

\newcommand{\QED}{\end{proof}}
\newcommand{\Proof}{\begin{proof}}

\newcommand{\id}{\on{id}}
\newcommand{\ba}{\begin{array}}
	\newcommand{\ea}{\end{array}}

\newcommand{\set}[2]{\left\{#1 \mid #2 \right\}}

\newcommand{\eqsub}{\begin{subequations}\begin{eqnarray}}
		\newcommand{\eneqsub}{\end{eqnarray}\end{subequations}}

\newcommand{\ol}{\overline}

\newdimen\ycell
\def\youngstyle{\textstyle}
\def\yboxit#1{{\setbox0=
		\hbox to\ycell{\hss$\mathsurround0pt\youngstyle#1$\hss}\relax
		\ht0=0.7\ycell\dp0=0.3\ycell\vtop{\vbox{\hrule
				\hbox{\vrule\box0\vrule}}\hrule}\kern-0.4pt}}
\def\yfin{\end}
\def\yline#1\cr{\def\donne{#1}\ifx\donne\yfin\def\nextyou{\kern0.4pt}\else
	\hbox{\let\nexty\spoolyline\spoolyline#1&\end&}\kern-0.4pt\fi\nextyou}
\def\spoolyline#1&{\def\donne{#1}\ifx\donne\yfin\def\nexty{\kern0.4pt}\else
	\yboxit{#1}\fi\nexty}
\def\young#1#2{\vbox{\ycell=#1\offinterlineskip
		\let\nextyou\yline\yline#2\cr\end\cr}}

\nc{\la}{\lambda}
\nc{\lam}{\lambda}
\nc{\U}[1][\g]{U_q(#1)}
\nc{\te}{\tilde{e}}
\nc{\tei}{\tilde{e}_i}
\nc{\tf}{\tilde{f}}
\nc{\tfi}{\tilde{f}_i}
\nc{\tU}{\widetilde U_q(\g)}
\nc{\tE}{\tilde{E}}
\nc{\tF}{\widetilde{\F}}
\nc{\tK}{\widetilde{K}}

\nc{\tk}{\tilde{k}}
\nc{\tkone}{\tk_{\ol{1}}}
\nc{\teone}{\tilde{e}_{\ol{1}}}
\nc{\tfone}{\tilde{f}_{\ol{1}}}

\nc{\teibar}{\tilde{e}_{\ol{i}}} \nc{\tfibar}{\tilde{f}_{\ol{i}}}
\nc{\tki}{{\tk}_{\ol {i}}}

\nc{\BZ}{{\mathbb{Z}}}
\nc{\al}{\alpha}
\nc{\qs}{{q}}
\nc{\lan}{\langle}
\nc{\ran}{\rangle}
\nc{\re}{{\mathrm{re}}}
\nc{\wt}{\operatorname{wt}}
\nc{\ch}{\operatorname{ch}}
\nc{\Um}[1][\g]{U^-_q(#1)}
\nc{\Ue}{U^+_q(\g)}
\nc{\eps}{\varepsilon}
\nc{\vphi}{\varphi}
\nc{\sphi}{\varphi^*}
\nc{\seps}{\varepsilon^*}

\nc{\nn}{\nonumber}

\nc{\vp}{\varpi}
\nc{\cls}{{\operatorname{cl}}}
\nc{\Wt}{{\operatorname{Wt}}}
\nc{\Us}{U'_q(\g)}
\nc{\La}{\Lambda}
\nc{\tLa}{\widetilde\Lambda}
\nc{\ro}{{\rm(}}
\nc{\rf}{{\rm)}}
\nc{\norm}{{\mathrm{norm}}}
\nc{\qbox}{\quad\mbox}
\nc{\braid}{{\mathfrak{B}}}
\nc{\Ad}{\operatorname{Ad}}
\nc{\Aut}{\operatorname{Aut}}
\nc{\dt}[1]{\tilde{\tilde #1}}
\nc{\Sn}{S^{{\mathrm{norm}}}}
\nc{\aff}{{\rm{aff}}}
\nc{\rk}{{\mathrm{rk}}}
\nc{\tP}{\widetilde{P}}
\nc{\tW}{\widetilde{W}}
\nc{\Dyn}{\mathrm{Dyn}}
\nc{\tD}{\widetilde{\Delta}}
\nc{\height}[1]{{\operatorname{ht}}(#1)}
\nc{\bl}{\bigl(}
\nc{\br}{\bigr)}
\nc{\Hecke}{\mathrm{H}}
\nc{\HA}{\Hecke^{\mathrm{A}}}
\nc{\HB}{\Hecke^{\mathrm{B}}}
\newcommand{\scbul}{{\,\raise1pt\hbox{$\scriptscriptstyle\bullet$}\,}}
\nc{\vac}{{\phi}}
\nc{\Bt}{\B_\theta(\g)}
\nc{\be}{\begin{enumerate}}
	\nc{\ee}{\end{enumerate}}
\nc{\low}{{\mathrm{low}}}
\nc{\upper}{{\mathrm{up}}}
\nc{\Zodd}{\Z_{\mathrm{odd}}}
\nc{\Ft}[1][n]{\mathbb{P}\mathrm{ol}_{#1}}
\nc{\Ftf}[1][n]{\widetilde{\mathbb{P}\mathrm{ol}}_{#1}}
\nc{\KA}{\on{K}^{\mathrm{A}}}
\nc{\KB}{\on{K}^{\mathrm{B}}}
\nc{\Res}{\on{Res}}
\nc{\Fc}[1][{n,m}]{\mathbf{F}_{#1}}
\nc{\tphi}{\tilde{\varphi}}
\nc{\CO}{\mathscr{O}}
\nc{\inte}{\mathrm{int}}
\nc{\Oint}{\mathcal{O}^{\ge0}_{\inte}}
\nc{\vs}{\vspace*}
\nc{\tLt}{\widetilde{L}}
\nc{\tL}{\widetilde{\Lambda}}
\nc{\tu}{\tilde{u}}
\nc{\noi}{\noindent}
\nc{\heigh}{\mathfrak{t}}
\nc{\lowest}{\mathfrak{l}}
\nc{\rootl}{\mathsf{Q}}
\nc{\cl}{\colon}
\nc{\uqpg}{U'_q(\mathfrak g)}
\nc{\uq}{\uqpg}
\nc{\Oh}{\widehat{\mathcal{O}}}
\nc{\pn}{p_{\mathfrak{n}}}

\nc{\KLR}{KLR algebra}
\nc{\KLRs}{KLR algebras}
\nc{\cor}{\mathbf{k}}
\nc{\cora}{{\cor(A)}}
\nc{\haut}{\mathrm{ht}}
\nc{\tens}{\mathop\otimes}
\nc{\gmod}{\mbox{-$\mathrm{gmod}$}}
\nc{\gMod}{\mbox{-$\mathrm{gMod}$}}
\nc{\proj}{\mbox{-$\mathrm{proj}$}}
\nc{\gproj}{\mbox{-$\mathrm{gproj}$}}
\nc{\smod}{\mbox{-$\mathrm{mod}$}}
\nc{\Mod}{\mbox{-$\mathrm{Mod}$}}
\nc{\h}{\mathfrak h}
\nc{\Rnorm}{R^{\rm{norm}}}

\nc{\Vhat}{\widehat{V}}
\nc{\F}{\mathcal{F}}

\def\T{{\mathcal T}}

\nc{\fd}[1][A]{\on{\mathrm{flat.dim}_{#1}}}
\nc{\bP}{{\mathbb{P}}}
\nc{\bPh}{\widehat{\mathbb{P}}}
\nc{\bK}[1][{n}]{\widehat{\mathbb{K}}_{#1}}
\nc{\bV}[1][{n}]{\widehat{V}^{\otimes{#1}}}
\nc{\bVK}[1][{n}]{\widehat{V}^{\otimes{#1}}_{\widehat{\mathbb{K}}}}
\nc{\hV}{\widehat{V}}
\nc{\opp}{\mathrm{opp}}
\nc{\col}{\colon}
\nc{\bnum}{\be[{\rm(i)}]}
\nc{\oep}{\epsilon}
\nc{\qtext}{\quad\text}
\nc{\qtextq}[1]{\quad\text{#1}\quad}
\nc{\longtwoheadrightarrow}[1][]{\xymatrix{\ar@{->>}[r]^-{{#1}}&}}
\nc{\epiTo}[1][]{\longtwoheadrightarrow[{#1}]}
\nc{\epito}{\twoheadrightarrow}
\nc{\monoTo}[1][]{\xymatrix{\ar@{>->}[r]^-{{#1}}&}}
\nc{\sym}{\mathfrak{S}}
\nc{\shojiH}{\mathcal H^{\natural}}
\nc{\inp}[1]{{({#1})_{\mathrm{n}}}}
\nc{\rtl}{\rootl}
\nc{\wtd}{\widetilde}
\nc{\etens}{\boxtimes}
\nc{\ds}[1]{\mathrm{d}(#1)}
\nc{\rmat}[1]{{\mathbf{r}}_%
	{\mspace{-2mu}\raisebox{-.6ex}{${\scriptstyle{#1}}$}}}
\nc{\rmats}[1]{{\mathbf{r}}_%
	{\mspace{-2mu}\raisebox{-.6ex}{${\scriptscriptstyle{#1}}$}}}
\nc{\shc}{\mathcal{C}}
\nc{\shs}{\mathcal{S}}
\nc{\Fct}{{\on{Fct}}}
\nc{\tC}{\widetilde{\shc}}
\nc{\Zp}{\Z_{\ge0}}
\nc{\tPhi}{\widetilde{\Phi}}
\nc{\tT}{{\widetilde{\T}}}
\nc{\Ob}{\on{Ob}}
\nc{\bwr}{\mbox{\large$\wr$}}
\nc{\Img}{\on{Im}}
\nc{\Ab}{\mathcal{A}^{\mathrm{big}}}
\nc{\Sb}{\mathcal{S}^{\mathrm{big}}}
\nc{\As}{\mathcal{A}}
\nc{\Ss}{\mathcal{S}}
\nc{\ntens}{\widetilde{\otimes}}
\nc{\hR}{\widehat{R}}
\nc{\nconv}{\mathop{\mbox{\large $\odot$}}}
\nc{\snconv}{\mbox{\scriptsize$\odot$}}
\nc{\ts}{\tilde{s}}
\nc{\sho}{\mathcal{O}}
\nc{\bc}{\begin{cases}}
	\nc{\ec}{\end{cases}}
\nc{\slnh}{{\widehat{\mathfrak{sl}}_N}}
\nc{\UA}{U_q'(\slnh)}
\nc{\KR}{R_K}
\nc{\cQ}{\mathcal{Q}}
\nc{\Irr}{\mathcal{I}rr}
\nc{\tQ}{\widetilde{\cQ}}
\nc{\bs}{\mathbf{s}}
\nc{\bL}{\mathbb{L}}
\nc{\tg}{\tilde{g}}

\nc{\conv}{\mathbin{\mbox{\large $\circ$}}}
\nc{\shconv}{\mathbin{\large\diamond}}
\nc{\hconv}{\mathbin{\mbox{\Large $\shconv$}}}

\nc{\Rm}{R^{\mathrm{ren}}}

\nc{\bQ}{\ol{Q}}

\nc{\de}{\on{\textfrak{d}}}

\nc{\xmono}{\ar@{>->}}
\nc{\xepi}{\ar@{->>}}
\nc{\db}[1]{\raisebox{-.5ex}[2ex][1.8ex]{$#1$}}
\nc{\wb}[1]{\mbox{$\rule[-1.1ex]{0ex}{2ex}#1$}}
\nc{\univ}{\mathrm{univ}}
\nc{\rM}{{}^*\mspace{-2mu}M}
\nc{\lM}{M^*}
\nc{\uqm}{\uq\smod}
\nc{\tR}{\widetilde{R}_{\gamma,\beta}}
\nc{\tx}{\tilde{x}}
\nc{\bi}{\mathbf{i}}
\nc{\ttau}{\widetilde{\tau}}

\nc{\tEnd}{\on{\widetilde{E}nd}}
\nc{\tHom}{\on{\widetilde{H}om}}

\nc{\K}{{J}}
\nc{\Kex}{{\K}_{\mathrm{ex}}}
\nc{\Kfr}{{\K}_{\mathrm{f\mspace{.01mu}r}}}
\nc{\coro}{\cor}
\nc{\tB}{\widetilde{B}}
\nc{\seed}{\mathscr{S}}

\nc{\up}{\mathrm{up}}
\nc{\bfa}{\mathbf{a}}

\newlength{\mylength}
\title[Modified Ariki-Koike algebra and Yokonuma-Hecke like relations]
{Modified Ariki-Koike algebra and Yokonuma-Hecke like relations}

\author{Myungho Kim, Sungsoon Kim$^\ast$}

\address[M. Kim]{Department of Mathematics, Kyung Hee University, Seoul 02447, South Korea}

\email[M. Kim]{mkim@khu.ac.kr}

\thanks{The research of M.\ Kim was supported by the National Research Foundation of
	Korea (NRF) Grant funded by the Korea government(MSIP)
	(NRF-2022R1F1A1076214 and NRF-2020R1A5A1016126).}

\address[S. Kim - Corresponding Author]{ D\'epartement de Math\'ematiques, CNRS UMR 7352 -UPJV, 80039 Amiens, France}

\email[S. Kim]{sungsoon.kim@u-picardie.fr}
\thanks{ ($\ast$) S. Kim, corresponding author, is grateful of KIAS for the host during this work as well as l'IEA-France for partial supports (2023).}

\date{August 6, 2024}

\nc{\Hnr}{\mathcal H_{n,r}}

\begin{document}
	\maketitle
	\tableofcontents

	\begin{abstract}
		We find new presentations of the modified Ariki-Koike algebra (known also as Shoji's algebra) $\mathcal H_{n,r}$  over an  integral domain $R$ associated with a set of parameters  $q,u_1,\ldots,u_r$ in $R$.   
		It turns out that the algebra $\mathcal H_{n,r}$  has a set of generators $t_1,\ldots,t_n$ and $g_1,\ldots g_{n-1}$ subject to some defining relations similar to the relations of   Yokonuma-Hecke algebra. 
		We also obtain a presentation of $\mathcal H_{n,r}$ which is independent of the choice of $u_1,\ldots u_r$.  
		As applications of the presentations,  we find an explicit and direct isomorphism between the modified Ariki-Koike algebras with different choices of parameters $(u_1,\ldots,u_r)$.
We also find an explicit trace form on the algebra $\mathcal H_{n,r}$ which is symmetrizing provided the parameters $u_1,\ldots, u_r$ are invertible in  $R$. 
	We show that the symmetric group $\sym(r)$ acts on the algebra $\Hnr$,  and find a basis and a set of generators of the fixed subalgebra $\Hnr^{\sym(r)}$.
	\end{abstract}

	\section{Introduction}

In the article \cite{Shoji}, Shoji introduced a new presentation of the generic  Ariki-Koike algebra of the  complex reflection group $G(r,1,n)$ consisting of $n\times n$ monomial matrices whose non-zero entries are $r$-th roots of unity. 
Since this algebra is defined over the ring $\Z[q, q\inv,  u_1,\ldots, u_r][\Delta^{-1}]$,  where $\Delta=\prod_{1\le i<j\le r}(u_i-u_j)$ and $q,u_1,\ldots,u_r$ are indeterminates,
it can be specialized to an algebra over an integral domain $R$,  i.e.,  the $R$-algebra with the same generators and relations with respect to the parameters $q,u_1,\ldots,u_r$ in $R$,  provided that $q$ and $\prod_{1\le i<j\le r}(u_i-u_j)$ are invertible in $R$. 
Its definition, given  in  Definition  $\ref{def:Hnr}$, is with the generators $\{T_i,t_j \, \vert \, 1\le i\le n-1, \ 1\le j\le n \}$ and the defining relations  involving the parameters $q,u_1,\ldots,u_r$.
	Here $ T_1,\ldots,T_{n-1} $  are the generators of  the Iwahori-Hecke algebra  of the symmetric group $\mathfrak S(n)$ and $t_1,\ldots,t_n$ satisfy a common relation of degree $r$. 
	
	
	This algebra with new presentation is called the \emph{Modified Ariki-Koike algebra},  or \emph{Shoji algebra},   and we denote it by $\Hnr(R,q,u_1,\ldots,u_r)$, 
	or simply $\Hnr$, if there is no risk of confusion. 
	The algebra $\Hnr$ has the following $R$-basis (\cite[Theorem 3.7]{Shoji})
	\eq \label{eq:basis1}
	\set{t_1^{c_1}\cdots t_n^{c_n}T_w}{w \in \sym(n), \  0\le c_1,\ldots,c_n \le r-1}
	\eneq
	thus $\Hnr$ is a free $R$-module of rank $\mid  G(r,1,n) \mid =r^nn!$.
	One advantage of this new presentation  of the generic Ariki-Koike algebra of $G(r,1,n)$ is that,  using the fact that the generators $\{t_j \, \vert \, 1\le j\le n-1\}$ are  symmetric in the defining relations,  one can embed $\Hnr\otimes \mathcal H_{m,r}$ into $\mathcal H_{m+n,r}$ as in the case of the group algebras.  It follows that  all the irreducible representations of $\Hnr$ can be constructed as induced modules from certain subalgebras in an  analogous way to the group case  (\cite{Shoji}). 
In  \cite{SawadaShoji}, the algebra $\Hnr(R,q,u_1,\ldots,u_r)$ is studied in connection with the Ariki-Koike algebra over $R$ with the same parameters $q,u_1,\ldots,u_r$ (see \cite[Section 1]{SawadaShoji}).  Even though they are not isomorphic in general, 
	there is an interesting algebra homomorphism from the latter to the former,  which allowed the authors in \cite{SawadaShoji} to produce  some connections  between their representations.  For example,  an estimate of the decomposition numbers for the Ariki-Koike algebra  of $G(r,1,n)$  is obtained in terms of the decomposition numbers for the modified Ariki-Koike algebra. 
	
	\bigskip
	
	In this paper,  we obtain a  new presentation of the modified Ariki-Koike algebra $\Hnr$ and study their applications.  
	 In fact,  our new presentation is similar to the standard presentation given by Juyumaya (see, for example \cite{Juyumaya}) of the \emph{Yokonuma-Hecke algebra}  (of type A) $\mathcal Y_{n,r}$,
where the Yokonuma-Hecke algebras  are algebras over $\C[q,q\inv]$
introduced by Yokonuma (\cite{YokonumaHecke} ) as a generalization of Iwahori-Hecke algebras.
 We call this presentation of $\Hnr$ the \emph{Yokonuma-Hecke like presentation of $\Hnr$}. 
 So naturally   some  properties analogous to the ones we had for $ \mathcal Y_{n,r}$ can be developed for $\Hnr$ as well.

	\smallskip

	The modified Ariki-Koike algebra  $\Hnr$ and the Yokonuma-Hecke algebra $\mathcal Y_{n,r}$ are closely related. We briefly recall the relationship following  \cite[Section 3.1]{ERH}.
In \cite{Lusztig,JLPdA}, it is shown that the Yokonuma-Hecke algebra $\mathcal Y_{n,r}$   is isomorphic to  a direct sum of matrix algebras over Iwahori-Hecke algebras of type $A$ as  $\C[q,q\inv]$-algebras.   On the other hand, in \cite{SawadaShoji, HuStoll}, it is shown that there is an isomorphism of $R$-algebras between the modified Ariki-Koike algebra $\Hnr$  and a direct sum of matrix algebras over Iwahori-Hecke algebras of type $A$ over $R$. 
Interestingly enough, when  $R=\C[q,q\inv]$, those two direct sums of matrix algebras are the same and hence  the algebras $\mathcal Y_{n,r}$ and $\Hnr$ are isomorphic to each other.
In  \cite[Theorem 13]{ERH}, an explicit isomorphism between the modified Ariki-Koike algebra $\mathcal H_{n,r}$ over $R=\C[q, q\inv]$ with the choice $u_k=e^{\frac{2\pi \sqrt{-1} k}{r}}$ $(1\le k \le r)$ and the Yokonuma-Hecke algebra $\mathcal Y_{n,r}$ is given. Recall that $\mathcal H_{n,r}$ acts faithfully on a tensor space $V^{\otimes n}$ (\cite{SakamotoShoji, Shoji, HuStoll}).
In \cite{ERH} the authors construct a faithful representation of $\mathcal Y_{n,r}$ on the tensor space $V^{\otimes n}$, and  show that the image of  $\mathcal Y_{n,r}$ in $\End(V^{\otimes n})$ is equal to the image of $\Hnr$ with the choice of the parameters above.  Hence  an explicit formula for the isomorphism is given by identifying the standard generators in both algebras with the operators on $V^{\otimes n}$ under the specialization of the parameters $u_k=e^{\frac{2\pi \sqrt{-1} k}{r}}$ $(1\le k \le r)$.

\smallskip
Our new presentation of $\mathcal H_{n,r}$ with \emph{arbitrary} choice of parameters can be understood as a generalization of the isomorphism in \cite[Theorem 13]{ERH}. 
More precisely, we obtain a new generating set and defining relations of $\Hnr$, which becomes the standard presentation of the Yokonuma-Hecke algebra $\mathcal Y_{n,r}$ if we take $R=\C[q,q\inv]$ and $u_k=e^{\frac{2\pi \sqrt{-1} k}{r}}$ $(1\le k \le r)$.

\smallskip

	As the first step to obtain this Yokonuma-Hecke like presentation of $\Hnr$, we prove that 
the following set
	$$\left\{b_{k_1,\ldots,k_n}:=\displaystyle\prod_{1\le i\le n} \prod_{1\le j \le r, \, j \neq k_i} \dfrac{t_i-u_j}{u_{k_i}-u_j} \,\vert \, (k_1,\ldots,k_n) \in [1,r]^n \right\}$$ 
is an $R$-basis	of the subalgebra $R[t_1,\ldots,t_n]$ generated by the mutually commuting generators $t_1,\ldots,t_n$. 
	This basis forms a complete system of orthogonal idempotents of the subalgebra $R[t_1,\ldots,t_n]$ (Lemma \ref{orthonormality}).  
We set 
	$$g_i := T_i-(q-q^{-1}) \sum_{k_i < k_{i+1}} b_{k_1,\ldots,k_n} \quad (1\le i\le n-1).$$
and we prove that $g_i$'s satisfy the braid relations and the relations  $g_it_i=t_{i+1}g_i$, $g_i t_j =t_j g_i$ $(|i-j|\ge2)$. 
	
Then by using the basis in \eqref{eq:basis1},   we obtain  the following  two new $R$-bases of $\Hnr$~:
	\eq \label{eq:basis2}
	\set{t_1^{c_1}\cdots t_n^{c_n}g_w}{w \in \sym(n), \  0\le c_i \le r-1}\quad \text{and}
	\eneq
	\eq\label{eq:basis3}
	\set{b_{k_1,\ldots,k_n}g_w}{(k_1,\ldots,k_n) \in [1,r]^n, \ w\in \sym(n)}.
	\eneq
	\noindent
	These bases allow us to obtain two new presentations of $\Hnr$, as written in Theorem \ref{thm:main}, and in Theorem \ref{thm:presentation2}, respectively.
Note that if $R=\C[q,q\inv]$ and $u_k=e^{\frac{2\pi \sqrt{-1} k}{r}}$ $(1\le k \le r)$, then the presentation in Theorem \ref{thm:main} becomes the standard presentation of $\mathcal Y_{n,r}$, and the presentation in Theorem \ref{thm:presentation2} becomes the presentation of $\mathcal Y_{n,r}$ appeared in \cite[Section 2]{JLPdA}.

	As an application of the Yokonuma-Hecke like presentation in Theorem \ref{thm:main},  we construct a trace form on  $\Hnr$ in Section 4. 
	Recall that an $R$-linear map $f:\mathcal  A \to R$ from an $R$-algebra $\mathcal A$ to its base ring $R$ is called a \emph{trace form} if
	$
	f(xy)=f(yx) \,\, \text{for}\,\, x,y \in\mathcal  A.
	$  
	A trace form $f$ is called \emph{symmetrizing} if the bilinear form 
	$\mathcal A \times \mathcal  A \to R$ given by $(x,y) \mapsto f(xy)$
	is non-degenerate.  
	The trace form $\tau : \mathcal H_{n,r} \longrightarrow R$ in our case is given by
	\eq \label{eq:traceform}
	\tau(t_1^{c_1}\cdots t_n^{c_n}g_w)=\begin{cases}
		1 & \text{if}\ \  w=\id_{\sym(n)}, c_1=\cdots=c_n=0
		\\
		0 & \text{ otherwise }
	\end{cases}
	\eneq
	for $0\le c_1,\ldots, c_r \le r-1$, $w \in \sym(n)$. 
	Note that this formula of $\tau$ is exactly the same as the one in  \cite[Proposition 10]{CLPdA} for the Yokonuma-Hecke algebra $\mathcal Y_{n,r}$.  Since we have the Yokonuma-Hecke like presentation,  the same proof as in \cite[Proposition 10]{CLPdA}  works for $H_{n,r}$ to show that $\tau$ is a trace form (Corollary 4.2).
	We further show that  $\tau$ is symmetrizing  under the condition that the product $\sigma_r:=u_1 \cdots u_r$ is invertible in $R$ (Corollary 4.6) by  constructing an explicit dual basis to the basis in \eqref{eq:basis2}.
	\medskip
	
	The other new presentation  of $\Hnr$, in the Theorem \ref{thm:presentation2}, is independent of the choice of parameters $u_1,\ldots, u_r$  
 and this fact comes from the observation that the parameters $u_1,\ldots, u_r$ do not appear in the presentation. Hence  the algebras $H_{n,r}(R, q, u_1,\ldots u_r)$  are isomorphic 
	no matter what values we choose for $u_1,\ldots, u_r$.  
 Moreover, the presentation enables us to provide an explicit isomorphism directly in Corollary \ref{cor:para_indep}. 
 Note that the independence on the parameters of the modified Ariki-Koike algebra follows  from the structure theorem of $\Hnr$ in \cite{SawadaShoji, HuStoll} by showing that they are isomorphic to the direct sum of the matrix algebra, hence our claim is not new.  
But we emphasize that our isomorphism in Corollary \ref{cor:para_indep}   is explicit and is given directly by using only the terms of the standard generators of $\Hnr$.

	\medskip
	
	Note that there is a natural action of the symmetric group $\sym(r)$ on $R[t_1,\ldots,t_n]$
	given by $^\sigma b_{k_1,\ldots,k_n} : = b_{\sigma(k_1),\ldots,\sigma(k_n)}$ for $\sigma  \in\sym(r)$.
	Then we  show that  $\sym(r)$ acts on the algebra $\Hnr$ by the presentation in Theorem \ref{thm:presentation2}.
	In  \cite{JLPdA, JLPdA17},   the authors introduce a basis  $\{E_\chi\}$ of the commutative subalgebra $\C[q, q\inv][t_1,\ldots,t_n]$ of the Yokonuma-Hecke algebra $\mathcal Y_{n,r}$ and show that $\sym(r)$ acts on the basis and the algebra $\mathcal Y_{n,r}$. Indeed the basis $\{b_{k_1,\ldots,k_n}\}$ and the presentation in Theorem \ref{thm:presentation2} can be regarded as a generalization of  $\{E_\chi\}$ and the presentation in \cite[Section 2.2]{JLPdA17} so that one can consider the fixed subalgebra $\Hnr^{\sym(r)}$ of the modified Ariki-Koike algebra as it was done in \cite{JLPdA17} for the Yokonuma-Hecke algebras. 
	Similarly as in \cite{JLPdA17},  we obtain an $R$-basis  and a set of generators of the fixed subalgebra.   In the case of the Yokonuma Hecke algebra with $n \ge r$,  the  fixed subalgebra is known as the \emph{the algebra of braids and ties} and denoted by $BT_n$ (see \cite[Section 4 ]{JLPdA17} and references therein).  
 It is interesting  that such  knot-theoretically compelling algebras appear not only as subalgebras of $\mathcal Y_{n,r}$, but also as subalgebras of $\Hnr$.  One may expect more connections between $\Hnr$ and knot theory.
	\medskip

	Lastly, this paper is organized as follows.  In Section 2, we recall the definition of the modified Ariki-Koike algebra $\Hnr$.  We  introduce and study the basis elements $b_{k_1,\ldots, k_n}$ of the commutative subalgebra of $\Hnr$ generated by $t_1,\ldots,t_n$. 
	In Section 3, we show various relations among the elements $g_j$'s and others.  Then we obtain the bases \eqref{eq:basis2}, \eqref{eq:basis3},  and the new presentations of $\Hnr$.
	In Section 4,  we define the trace form \eqref{eq:traceform} and show that it is non-degenerate if $u_i$'s are all invertible in $R$.
 In Section 5,  we study the fixed subalgebra of the action of $\sym(r)$ on $\Hnr$.

	\medskip
{\bf Acknowledgments.}\ 
We thank Toshiaki Shoji for a valuable communication and referees for their precious comments and suggestions on the article.
	
	
	\section{Modified Ariki-Koike algebras}
	Let $R$ be an integral domain
	and $R^\times$ be the group of invertible elements in $R$.

	We take 
	$q , u_1,\ldots,u_r \in R$ such that
	\eq
	q \in R^\times  \quad \text{and} \quad
	\Delta:=\prod_{i>j} (u_i-u_j) \in  R^\times.
	\eneq

	\Def \label{def:Hnr}
	The \emph{modified Ariki-Koike algebra (also called the Shoji's algebra) } $\mathcal H_{n,r}=\mathcal H_{n,r}(R,q,u_1,\ldots,u_r)$ is an associative $R$-algebra generated by
	$$t_1,\ldots,t_n, T_1,\ldots,T_{n-1} $$
	with the defining relations
	
	\eq
	&& (T_i-q)(T_i+q^{-1})=0 \qquad  (1 \le i \le n-1) \label{rel:quad} \\
	&& (t_i-u_1)\cdots(t_i-u_r)=0 \qquad (1 \le i \le n)\\
	&& T_iT_{i+1}T_i=T_{i+1}T_iT_{i+1} \qquad (1 \le i \le n-2)\\
	&& T_iT_j=T_jT_i \qquad (|i-j| \ge 2)\\
	&& t_it_j=t_jt_i \qquad (1 \le i,j \le n)\\
	&& T_j t_k= t_k T_j \qquad  (k\neq j,j+1)\\
	&& T_{j-1} t_j= t_{j-1} T_{j-1} + \Delta^{-2} \displaystyle\sum_{1\le  c_1<c_2\le r} (u_{c_2}-u_{c_1})(q-q^{-1}) F_{c_1}(t_{j-1})F_{c_2}(t_{j})\\
	&& T_{j-1} t_{j-1}= t_{j} T_{j-1} - \Delta^{-2} \displaystyle\sum_{1 \le c_1<c_2 \le r} (u_{c_2}-u_{c_1})(q-q^{-1}) F_{c_1}(t_{j-1})F_{c_2}(t_{j}),
	\eneq
	
	where
	$\Delta=\displaystyle\prod_{i>j}(u_i-u_j)$, and $F_i(X) \in R[X]$ is the polynomial uniquely determined by the conditions that $\deg(F_i)=r-1$ and $F_{c}(u_{c'}) = \delta_{c,c'}\Delta$ for $1 \le c,c' \le r$.
\end{definition}
If we write $F_i(X)=\sum_{j=1}^r h_{i,j} X^{j-1}$,  then the matrix $H(u_1,\ldots,u_r):=(h_{i,j})_{1\le i,j\le r}$ is given by $H(u_1,\ldots,u_r) = \Delta V(u_1,\ldots,u_r)^{-1}$,  the adjugate matrix of the Vandermonde matrix $V(u_1,\ldots,u_r)$ where
\eq \label{eq:Vandermonde}
V(u_1,\ldots,u_r)_{i,j}=u_j^{i-1} \text{ for}  \quad 1\le i,j \le r.
\eneq

Note that 
\eq \label{eq:t^r}
t_i^r=\sum_{k=0}^{r-1} (-1)^{r-k+1} \sigma_{r-k} \, t_i^{k} 
\eneq
for each $1\le i\le n$, where $\sigma_r$ denotes the $r$-th elementary symmetric polynomial in $u_1,\ldots, u_r$. 
\begin{remark}
	\begin{enumerate}
		\item  Note that the algebra in \cite{Shoji} is the case in Definition \ref{def:Hnr} when $R=\Z[q,q^{\inv,}  u_1,\ldots,u_r,  \Delta] \subset \Q(q,u_1,\ldots,u_r)$,  where $q,u_1\ldots,u_r$ are indeterminates and $\Q(q,u_1,\ldots,u_r)$ denotes the ring of rational functions in them.  In this paper,  we consider the specializations of the one  in \cite{Shoji} following   \cite[Section 1.2]{SawadaShoji}.
		\item The generators $t_i$ and  $T_{j-1}$ correspond to $\xi_i$ and  $a_j$ in \cite{Shoji}, respectively.
	\end{enumerate}
\end{remark}

Let $\sym(n)$ be the symmetric group of $n$-letters.
The following fundamental result is proved in \cite{Shoji}.  
\begin{theorem}\cite[Theorem 3.7]{Shoji} \label{thm:Tbasis}
	The algebra $\mathcal H_{n,r}$  has an $R$-basis %
	\eq \label{eq:basis}
	\set{t_1^{c_1}\cdots t_n^{c_n}T_w}{w \in \sym(n), \  0\le c_1,\ldots,c_n \le r-1}.
	\eneq
\end{theorem}

\medskip

It follows that
the subalgebra generated by  $T_1,\ldots T_{n-1}$ is isomorphic to the Iwahori-Hecke algebra of the symmetric group $\sym(n)$. 
The subalgebra $R[t_1,\ldots,t_n]$ generated by $t_1,\ldots, t_n$ is isomorphic to  the quotient ring $R[X_1,\ldots X_n]/(\set{\prod_{j=1}^r(X_i-u_j)}{1\le i\le n})$ of the polynomial ring $R[X_1,\ldots X_n]$.
Note that $R[t_1,\ldots,t_n]$ is a free $R$-module with a basis $\set{t_1^{c_1}\cdots t_n^{c_n}}{  0\le c_1,\ldots, c_n \le r-1}$.

\subsection{The subalgebra $R[t_1,\ldots,t_n]$}\label{formulasforbasisB}\
Set
\eqn
[1,r]:=\set{a \in \Z}{1\le a\le r} \quad \text{and} \quad [1,r]^n= \overbrace{[1,r]\times \cdots \times[1,r]}^{n-\text{times}}.
\eneqn

For each $(k_1,\ldots k_n) \in [1,r]^n$, define
\eqn
&&\tilde b_{k_1,\ldots,k_n}:=\displaystyle\prod_{1\le i\le n} \prod_{1\le j \le r, \, j \neq k_i} \dfrac{X_i-u_j}{u_{k_i}-u_j} \in R[X_1,\ldots,X_n] \quad \text{and}\\
&&b_{k_1,\ldots,k_n}:=\displaystyle\prod_{1\le i\le n} \prod_{1\le j \le r, \, j \neq k_i} \dfrac{t_i-u_j}{u_{k_i}-u_j} \in R[t_1,\ldots,t_n].
\eneqn
In other words, the element  $\tilde b_{k_1,\ldots,k_n}$ is the  polynomial in $X_1,\ldots, X_n$ of degree $r-1$ for each $X_i$ such that $\tilde b_{k_1,\ldots,k_n}(u_{j_1},\ldots,u_{j_n})=\delta _{(k_1,\ldots,k_n),(j_1,\ldots,j_n)}$  for any $(j_1,\ldots,j_n) \in [1,r]^n$  and
$b_{k_1,\ldots,k_n}$ is the image of $\tilde b_{k_1,\ldots,k_n}$ in the quotient ring $R[t_1,\ldots,t_n]$.

\begin{lemma}\label{orthonormality}
	The set
	$$\set{b_{k_1,\ldots,k_n} \in R[t_1,\ldots,  t_n]}{(k_1,\ldots,k_n)\in [1,r]^n}$$
	forms  an $R$-basis of $R[t_1,\ldots,t_n]$.
	It is  a complete system  of orthogonal idempotents of $R[t_1,\ldots,t_n]$.
\end{lemma}
\begin{proof}
	Note that
	\eqn
	t_i \prod_{1\le j\le r, \,  j\neq c} (t_i-u_j)=u_c \prod_{1\le j\le r, \, j\neq c} (t_i-u_j)
	\eneqn
	for any $1\le i\le n$, $1\le c \le r$. 
	Hence we have
	$(t_i-u_{k_i})b_{k_1,\ldots,k_n}=0$
	so that
	\eq
	f(t_1,\ldots, t_n)  b_{k_1,\ldots, k_n} =f(u_{k_1}, \ldots, u_{k_n})   b_{k_1,\ldots, k_n} \label{actionti}
	\eneq
	for any $f\in R[X_1,\ldots,X_n]$ and $(k_1,\ldots, k_n) \in [1,r]^n$.
	It follows that  $b_{k_1,\ldots,k_n}$'s are idempotents and orthogonal to each other. Thus they are linearly independent over $R$ and hence
	form  an $R$-basis of $R[t_1,\ldots,t_n]$, which is a free $R$-module of rank $r^n$.
	It follows that
	\eqn 
	f(t_1,\ldots,t_n)=\sum_{(k_1,\ldots,k_n) \in [1,r]^n} f(u_{k_1},\ldots, u_{k_n}) b_{k_1,\ldots, k_n}
	\eneqn
	for any polynomial $f(X_1,\ldots,X_n)$. In particular we have $1 = \sum_{(k_1,\ldots,k_n) \in [1,r]^n}  b_{k_1,\ldots, k_n}$, as desired.
\end{proof}

Set
\eqn
B_{j-1}:=-\Delta^{-2} \displaystyle\sum_{c_1<c_2} (u_{c_2}-u_{c_1})(q-q^{-1}) F_{c_1}(t_{j-1})F_{c_2}(t_{j})
\eneqn
so that
\eqn\label{Tjti}
&T_{j-1} t_j= t_{j-1} T_{j-1} -B_{j-1} \quad \text{and}\quad
&T_{j-1} t_{j-1}= t_{j} T_{j-1} +B_{j-1}.
\eneqn

Let $\widetilde B_{j-1} :=-\Delta^{-2} \displaystyle\sum_{c_1<c_2} (u_{c_2}-u_{c_1})(q-q^{-1}) F_{c_1}(X_{j-1})F_{c_2}(X_{j}). $
Then
\eqn
&&\widetilde B_{j-1} (u_{k_1},\ldots,u_{k_n}) = -\Delta^{-2} \displaystyle\sum_{c_1<c_2} (u_{c_2}-u_{c_1})(q-q^{-1}) \delta_{c_1=k_{j-1}, c_2=k_j} \Delta^2 \\
&&= \begin{cases}
	(q-q^{-1})(u_{k_{j-1}}-u_{k_j}) & \text{if} \ k_{j-1} < k_j, \\
	0 & \text{otherwise}.
\end{cases}
\eneqn
It follows that
\eq
B_{j-1}=(q-q^{-1}) \sum_{(k_1,\ldots, k_n) \in [1,r]^n, \ k_{j-1} < k_j } (u_{k_{j-1}}-u_{k_j}) b_{k_1,\ldots, k_n}.
\eneq

\begin{lemma}\label{relationTiwithBasis} 
	For $1\le p \le n-1$, we have
	\eqn
	&&T_p b_{k_1,\ldots,k_{p},k_{p+1},\ldots,k_n} - b_{k_1,\ldots,k_{p+1},k_p,\ldots,k_n} T_p
	\\
	&&=(q-q^{-1} ) \left(\delta(k_{p} < k_{p+1}) b_{k_1,\ldots,k_{p},k_{p+1},\ldots,k_n} - \delta(k_{p} > k_{p+1}) b_{k_1,\ldots,k_{p+1},k_{p},\ldots,k_n}  \right).
	\eneqn
\end{lemma}
\begin{proof}
	First, we will prove when $n=2$ and $p=1$ case. Note that
	\eq
	T_1(t_1-u)(t_2-u') =(t_2-u)(t_1-u')T_1+B_1(u-u')
	\eneq
	for any $u,u'\in R$.
	
	Hence we have
	\eqn
	&&T_1b_{x,y} = \left(\prod_{1\le j \le r, \, j \neq x} \dfrac{t_1-u_j}{u_{x}-u_j}\right) \left(\prod_{1\le j \le r, \, j \neq y} \dfrac{t_2-u_j}{u_{y}-u_j}\right) \\
	&&=T_1 \,\left(\prod_{1\le j \le r, \, j \neq x,y} \left(\dfrac{t_1-u_j}{u_{x}-u_j}\right) \left(\dfrac{t_2-u_j}{u_{y}-u_j}\right) \right) \left(\left(\dfrac{t_1-u_y}{u_{x}-u_y}\right)\left(\dfrac{t_2-u_x}{u_{y}-u_x}\right) \right)^{\delta(x\neq y)} \\
	&&=\left(\prod_{1\le j \le r, \, j \neq x,y} \left(\dfrac{t_1-u_j}{u_{x}-u_j}\right) \left(\dfrac{t_2-u_j}{u_{y}-u_j}\right) \right) \, T_1 \, \left(\left(\dfrac{t_1-u_y}{u_{x}-u_y})\right)\left(\dfrac{t_2-u_x}{u_{y}-u_x}\right) \right)^{\delta(x\neq y)} \\
	&&=\begin{cases}
		\left(\displaystyle\prod_{1\le j \le r, \, j \neq x,y} \left(\dfrac{t_1-u_j}{u_{x}-u_j}\right) \left(\dfrac{t_2-u_j}{u_{y}-u_j}\right) \right)  \left(\left(\dfrac{t_2-u_y}{u_{x}-u_y}\right) \left(\dfrac{t_1-u_x}{u_{y}-u_x}\right)\, T_1 +\dfrac{B_1}{u_x-u_y} \right)  & \text{if } \ x\neq y\\
		\left(\displaystyle\prod_{1\le j \le r, \, j \neq x} \left(\dfrac{t_1-u_j}{u_{x}-u_j}\right) \left(\dfrac{t_2-u_j}{u_{y}-u_j}\right) \right)  T_1
		& \text{if } \ x= y.
	\end{cases}
	\eneqn

	On the  other hand, we have
	\eqn
	&&b_{y,x} T_1 = \left(\prod_{1\le j \le r, \, j \neq y} \dfrac{t_1-u_j}{u_{y}-u_j}\right) \left(\prod_{1\le j \le r, \, j \neq x} \dfrac{t_2-u_j}{u_{x}-u_j}\right) \, T_1 \\
	&&=\left(\prod_{1\le j \le r, \, j \neq x,y} \left(\dfrac{t_1-u_j}{u_{y}-u_j}\right) \left(\dfrac{t_2-u_j}{u_{x}-u_j}\right) \right) \left((\dfrac{t_1-u_x}{u_{y}-u_x})(\dfrac{t_2-u_y}{u_{x}-u_y}) \right)^{\delta(x\neq y)}
	\, T_1.
	\eneqn
	
	Hence
	\eqn
	T_1 \, b_{xy} - b_{y,x} \,  T_1 =
	\delta(x\neq y)\left(\displaystyle\prod_{1\le j \le r, \, j \neq x,y} \left(\dfrac{t_1-u_j}{u_{x}-u_j}\right) \left(\dfrac{t_2-u_j}{u_{y}-u_j}\right)\right)  \dfrac{B_1}{u_x-u_y}
	\eneqn
	Set
	\eqn
	G(X_1,X_2):=\displaystyle\prod_{1\le j \le r, \, j \neq x,y} \left(\dfrac{X_1-u_j}{u_{x}-u_j}\right) \left(\dfrac{X_2-u_j}{u_{y}-u_j}\right)  \in R[X_1,X_2].
	\eneqn
	Then we have
	\eqn
	\begin{cases}
		G(u_a,u_b)=0 & \text{if } \ a \neq x,y \ \text{or} \ b\neq x,y, \\
		G(u_x,u_y)=G(u_y,u_x)=1.
	\end{cases}
	\eneqn
	Hence
	\eqn
	G(t_1,t_2) =  a_{x,x}b_{x,x} +a_{y,y}b_{y,y} + b_{x,y}+b_{y,x}
	\eneqn
	for some $a_{x,x}, a_{y,y} \in R$.
	Thus
	\eqn
	G(t_1,t_2) \dfrac{B_1}{u_x-u_y}  &&= G(t_1,t_2) \left(\dfrac{q-q^{-1}}{u_x-u_y}  \right) \left(\sum_{1\le \al < \beta\le r} (u_\al-u_\beta) b_{\al,\beta}\right) \\
	&&=\left(\dfrac{q-q^{-1}}{u_x-u_y}  \right) \left(   \delta(x<y )(u_x-u_y ) b_{x,y}+   \delta(x>y )(u_y-u_x )b_{y,x}  \right) \\
	&&=(q-q^{-1}) \left( \delta(x<y ) b_{x,y}- \delta(x>y ) b_{y,x}  \right),
	\eneqn
	as desired.
	
	Now assume that $n \ge 2$ and $1 \le p \le n-1$.
	Note that $$\prod_{i\neq p,p+1} \left( \prod_{1\le j \le r, \ k_i \neq j} \dfrac{t_i-u_j}{u_{k_i}-u_j} \right)$$
	commutes with $T_p$.
	Hence  by the above calculation for $n=2, p=1$ case, we have
	\noindent
	\eqn
	&&T_p b_{k_1,\ldots,k_{p},k_{p+1},\ldots,k_n} - b_{k_1,\ldots,k_{p+1},k_p,\ldots,k_n} T_p
	\\
	&&=(q-q^{-1} ) \left(\delta(k_{p} < k_{p+1}) b_{k_1,\ldots,k_{p},k_{p+1},\ldots,k_n} - \delta(k_{p} > k_{p+1}) b_{k_1,\ldots,k_{p+1},k_{p},\ldots,k_n}  \right)
	\prod_{i\neq p,p+1} \left( \prod_{1\le j \le r, \ k_i \neq j} \dfrac{t_i-u_j}{u_{k_i}-u_j} \right).
	\eneqn
	Since
	$$
	\dfrac{t_i-u_j}{u_{k_i}-u_j}b_{k_1,\ldots,k_{p},k_{p+1},\ldots,k_n}=b_{k_1,\ldots,k_{p},k_{p+1},\ldots,k_n} \quad \text{and} \quad
	\dfrac{t_i-u_j}{u_{k_i}-u_j}b_{k_1,\ldots,k_{p+1},k_{p},\ldots,k_n}=b_{k_1,\ldots,k_{p+1},k_{p},\ldots,k_n}
	$$
	for any $i\neq k_p,k_{p+1}$ and $1\le j\le r$ with $k_i \neq j$,
	we get the desired result.
\end{proof}

\subsection{Actions of symmetric groups on $R[t_1,\ldots,t_n]$}

Note that the symmetric group $\sym(n)$ acts on $R[t_1,\ldots,t_n]$ as $R$-algebra automorphisms by 
$$w.t_i:=t_{w(i)} \qquad \text{ for $1\le i\le n$, $w \in \sym(n)$ }.$$
It  also acts on the set $[1,r]^n$ by the place permutations $w.(k_1,\ldots,k_n):=(k_{w^{-1}(1)}, \ldots, k_{w^{-1}(n)})$.
Then we have
\eq
w.b_{(k_1,\ldots,k_n)} = b_{w.(k_1,\ldots,k_n)}\qquad \text{for} \ w \in \sym(n), \ (k_1,\ldots,k_n) \in [1,r]^n.
\eneq

On the other hand, the symmetric group $\sym(r)$ acts on $[1,r]^n$ by 
$$^\sigma (k_1,\ldots,k_n) := (\sigma(k_1),\ldots,\sigma(k_n)) \qquad \text{ for $1\le k_i \le r$, $\sigma \in \sym(r)$},$$
which induces an action on the basis given by
\eq \label{eq:sym_r action}
^\sigma b_{k_1,\ldots,k_n} : = b_{\sigma(k_1),\ldots,\sigma(k_n)}. 
\eneq
Since 
$${^\sigma b_{k_1,\ldots,k_n}} {^\sigma b_{k'_1,\ldots,k'_n} } = \delta_{(k_1,\ldots,k_n), (k'_1,\ldots,k'_n)} {^\sigma b_{k_1,\ldots,k_n}}, $$
the $R$-linear map
$\sigma$ extends to an $R$-algebra automorphism and  hence the group $\sym(r)$ acts on $R[t_1,\ldots,t_n]$.

Note that the actions of $\sym(n)$ and  $\sym(r)$ on $[1,r]^n$ are commuting to each other.

\begin{remark} \label{rem:OP}
	Let $OP_r(n)$ be the set of sequences $(I_1,\ldots,I_r)$ in subsets of $[1,n]$ such that $I_i\cap I_j=\emptyset$ for $i\neq j$ and $\displaystyle \bigsqcup_{ 1 \le i\le r} I_i = [1,n]$.
	Then there is a bijection between the sets $ [1,r]^n$ and $OP_r(n)$ given by
	\eqn
	\Psi: [1,r]^n &\to& OP_r(n) \\
	f &\mapsto & (f\inv(1),\ldots,f\inv(r)),
	\eneqn
	where we identify the set $[1,r]^n$ with the set of functions from $[1,n]$ to $[1,r]$.
	Then the action of $\sym(n)\times \sym(r)$ on $[1,r]^n$ is nothing  but
	\eqn
	w.(^\sigma f)=^\sigma(w. f)= \sigma \circ f \circ w^{-1} \qquad w \in \sym(n), \quad \sigma \in \sym(r),
	\eneqn 
	and $\Psi$ commutes with the action of $\sym(n) \times \sym(r) $ in \cite[Section 2.2]{JLPdA17}, in which $\sym(r)$ acts on $OP_r(n)$ as place permutations.
	Thus the set of orbits of $[1,r]^n$ under the action of $\sym(r)$ is in bijection with the set of partitions of $n$ at most $r$-many parts.
	More precisely, we have
	\eq \label{eq:orbits}
	\hskip -2em
	(\ell_1,\ldots,\ell_n) \in [k_1,\ldots,k_n] \ \quad \text{if and only if} \quad (\ell_i=\ell_j \Leftrightarrow k_i= k_j \quad \text{for all} \ 1 \le i,  j\le n),
	\eneq
	where 
	$[k_1,\ldots,k_n]$ denote the orbit of $(k_1,\ldots,k_n)$ under the action of $\sym(r)$. 
\end{remark}

Let $s_p\,\, (1\le p\le n-1)$ be the simple transposition in the symmetric group $\sym(n)$ permuting $p$ and $p+1$.
\begin{lemma}  \label{Ti_symm_pol} 
	Let $a \in R[t_1,\ldots,t_n]$ and $1\le p\le n-1$.
	If $s_p.a=a$, 
	then  we have
	\eqn
	T_pa=aT_p.
	\eneqn
	
\end{lemma}
\begin{proof}
	Write 
	$$a= \sum_{(k_1,\ldots,k_n)\in[1,r]^n} a_{(k_1,\ldots,k_n)} b_{k_1,\ldots,k_n}$$
	for some $a_{(k_1,\ldots,k_n)} \in R$.
	Since $s_p.a =a$, we have
	$$a_{(k_1,k_p, k_{p+1},\ldots,  k_n)} =a_{(k_1,\ldots,k_{p+1}, k_p,\ldots, k_n)} $$
	for all $(k_1,\ldots,k_n)\in[1,r]^n$.
	Then by Lemma \ref{relationTiwithBasis}, 
	we have 
	\eqn
	(q-q\inv)\inv (T_p a - s_p.a \, T_p )&&=\sum_{k_p<k_{p+1}} a_{(k_1,\ldots,k_n)} b_{k_1,\ldots,k_n}-\sum_{k_{p+1} < k_p}a_{(k_1,\ldots,k_n)} b_{s_p.(k_1,\ldots,k_n)} \\
	&&=\sum_{k_p<k_{p+1}} a_{(k_1,\ldots,k_n)} b_{k_1,\ldots,k_n}-\sum_{k_p < k_{p+1}}  a_{s_p.(k_1,\ldots,k_n)} b_{k_1,\ldots,k_n} =0,
	\eneqn
	as desired.
\end{proof}

\section{Yokonuma-Hecke like presentation}

For each $1\le i,j \le n$, define
\eq
B'_{i,j} :=-(q-q^{-1}) \sum_{k_i < k_j} b_{k_1,\ldots,k_n}.
\eneq
Note that
\eq\label{BiBi'}
-(t_{i}-t_{i+1} )B_{i,i+1}'= B_{i}\quad \text{for} \ 1\le i\le n-1.
\eneq
We also define for  each $1\le i\le n-1$ 
\eq
e_{i}:= \sum_{k_i = k_{i+1}} b_{k_1,\ldots,k_n}.
\eneq

\noindent
For each $1\le i\le n-1$ 
we define
\eq \label{eq:giTi}
g_i := T_i+B'_{i,i+1}.
\eneq

\begin{remark}\label{rem:def_gi}
	The above element $g_i$ is denoted by $S_{j+1}$ in \cite{Shoji} . It corresponds to $\bold G_i$  in \cite[Definition 5]{ERH}.
\end{remark}

\begin{example}
	Let $n=3$, $r=2$. Then the group $G(2,1,3)$ is the Weyl group of type $B_3$.  
	We have 
	$$\begin{cases}
		& b_{111}=\dfrac{t_1-u_2}{u_1-u_2}\cdot \dfrac{t_2-u_2}{u_1-u_2}\cdot \dfrac{t_3-u_2}{u_1-u_2},\qquad b_{121}=\dfrac{t_1-u_2}{u_1-u_2}\cdot\dfrac{t_2-u_1}{u_2-u_1}\cdot\dfrac{t_3-u_2}{u_1-u_2}, \\
		& b_{112}=\dfrac{t_1-u_2}{u_1-u_2}\cdot\dfrac{t_2-u_2}{u_1-u_2}\cdot\dfrac{t_3-u_1}{u_2-u_1},\qquad b_{122}=\dfrac{t_1-u_2}{u_1-u_2}\cdot \dfrac{t_2-u_1}{u_2-u_1}\cdot \dfrac{t_3-u_1}{u_2-u_1},\\
		& b_{211}=\dfrac{t_1-u_1}{u_2-u_1}\cdot \dfrac{t_2-u_2}{u_1-u_2}\cdot \dfrac{t_3-u_2}{u_1-u_2}, \qquad b_{221}=\dfrac{t_1-u_1}{u_2-u_1}\cdot \dfrac{t_2-u_1}{u_2-u_1}\cdot \dfrac{t_3-u_2}{u_1-u_2},\\
		& b_{212}=\dfrac{t_1-u_1}{u_2-u_1}\cdot \dfrac{t_2-u_2}{u_1-u_2}\cdot \dfrac{t_3-u_1}{u_2-u_1},\qquad b_{222}=\dfrac{t_1-u_1}{u_2-u_1}\cdot \dfrac{t_2-u_1}{u_2-u_1}\cdot \dfrac{t_3-u_1}{u_2-u_1},
	\end{cases}$$
	$$\begin{cases}
		& B'_{12}=-(q-q\inv)\{b_{121}+b_{122}\},\\
		& B'_{13}=-(q-q\inv)\{b_{112}+b_{122}\},\\
		& B'_{23}=-(q-q\inv)\{b_{112}+b_{212}\},
	\end{cases}$$
	and
	$$\begin{cases}
		& e_1=b_{111}+b_{112}+b_{221}+b_{222},\qquad e_2=b_{111}+b_{211}+b_{122}+b_{222},\\
		& g_1=T_1+B'_{12},\qquad g_2=T_2+B'_{23}.
	\end{cases}$$
	
\end{example}

\Prop\label{actiongi}
For $1\le i\le n$, $1\le j\le n-1$, we have
\eq g_j t_i=t_{s_j(i)} g_j.
\eneq
Hence for 
any $f \in R[t_1,\ldots,t_n], \ 1\le j \le n-1$,
we have
\eq
g_j f=(s_j.f) g_j. 
\eneq 
\end{prop} 
\begin{proof} We recall the following relations  from Definition \ref{def:Hnr}:
\eqn T_j t_k= t_k T_j \ (k\neq j,j+1), \qquad 
T_{j} t_{j+1}= t_{j} T_{j} -B_{j}, \quad \text{and} \quad
T_{j} t_{j}= t_{j+1} T_{j} +B_{j}.
\eneqn


For the case $i\neq j,j+1$, there is nothing to prove since $s_j(i)=i$. 

For $i=j+1$, by using the formula (\ref{BiBi'}), we get 
\eqn
&g_jt_{j+1}= (T_j+B'_{j,j+1})t_{j+1}=  t_{j} T_{j} -B_{j} +B'_{j,j+1}t_{j+1}=t_jT_j+(t_{j}-t_{j+1} )B_{j,j+1}'+ B'_{j,j+1}t_{j+1}\\
&= t_j(T_j+B_{j,j+1}')=t_jg_j=t_{s_j(i)}g_j
\eneqn

Similarly, for the case $i=j$, we obtain
\eqn
&g_jt_{j}= (T_j+B'_{j,j+1})t_{j}=  t_{j+1} T_{j}+B_{j} +B'_{j,j+1}t_{j}=t_{j+1}T_j-(t_{j}-t_{j+1} )B_{j,j+1}'+ B'_{j,j+1}t_{j}\\
&= t_{j+1}(T_j+B_{j,j+1}')=t_{s_j(i)}g_j,
\eneqn
as desired.
\end{proof}

\begin{lemma}\label{lem:relations}
We have
\eq \label{gkBij}
g_k B'_{i,j} =  B'_{s_k(i),s_k(j)} g_k,
\eneq
\eq\label{TkBij}
T_k B'_{i,j} = B'_{s_k(i),s_k(j)} T_k + (B'_{s_k(i),s_k(j)}- B'_{i,j})B'_{k,k+1}
\eneq
for $1\le k \le n-1$, $1\le i,j\le n$.
\end{lemma}
\begin{proof}
Note that for any $1\le i,j, \le n$,  $w\in \sym(n)$, we have
\eq
w.B'_{i,j}
=(q-q^{-1}) \sum_{k_i < k_j} b_{ (k_{w^{-1}(1)},\ldots,k_{w^{-1}(n)})}
=(q-q^{-1}) \sum_{k_{w(i)} < k_{w(j)}} b_{ (k_1,\ldots,k_n)}
=B'_{w(i),w(j)}.
\eneq
Thus Proposition \ref{actiongi} implies \eqref{gkBij}.
The second formula results immediately from the first one.
\end{proof}

\begin{prop}\label{prop:quadratic}
For $1\le i\le n-1$, we have
\eq
g_i^2 = 1+(q-q\inv)e_ig_i.
\eneq
\end{prop}
\begin{proof}
We have
$$g_i^2=(T_i+B'_{i,i+1})^2=T_i^2+T_iB'_{i,i+1}+B'_{i,i+1}T_i+{B'_{i,i+1}}^2$$
$$=\left( (q-q\inv)T_i + 1 \right )+\left(B'_{i+1,i}T_i+(B'_{i+1,i}-B'_{i,i+1})B'_{i,i+1}\right)+B'_{i,i+1}T_i+{B'_{i,i+1}}^2,$$
by the quadratic relation on $T_i$ and  Lemma \ref{lem:relations}.
Then by the fact that $\sum b_{k_1\cdots k_n}=1$ and the definition of $B'_{i,i+1}$, it becomes
$$g_i^2 =(q-q\inv)\left(\sum b_{k_1\cdots k_n} - \sum_{k_{i+1}<k_i}b_{k_1\cdots k_n}  -\sum_{k_i<k_{i+1}} b_{k_1\cdots k_n} \right)T_i+1+ B'_{i+1,i}B'_{i,i+1}.$$
We know that $B'_{i+1,i}B'_{i,i+1}=0$, thus we obtain the desired result
$$g_i^2=(q-q\inv)e_i \, T_i +1=(q-q\inv)e_ig_i+1,$$
by using $e_iT_i=e_i(g_i-B'_{i,i+1})=e_ig_i$ because $e_iB'_{i,i+1}=0.$
\end{proof}

The below relations can be verified in the faithful representation of $\Hnr$ constructed by Shoji (see \cite[Lemma 3.7]{SakamotoShoji}).  We  include here a direct proof using the defining relations.
\begin{prop}\label{prop:braid}
The following relations hold. 
\be\label{eq:presentationproperties}
\item $g_ig_j=g_jg_i$ for $|i-j| >1$,
\item $g_ig_{i+1}g_i = g_{i+1}g_ig_{i+1}$ for $1\le i \le n-1$.
\ee
\end{prop}
\Proof
(1) If $|i-j| >1$,  we know that $T_iT_j=T_jT_i$. 
Then the difference $g_ig_j-g_jg_i$ becomes
$$g_ig_j-g_jg_i= T_iB'_{j,j+1}+B'_{i,i+1}T_j-T_jB'_{i,i+1}-B'_{j,j+1}T_i=0,$$
where the last equality follows from Lemma \ref{lem:relations}.
\medskip

(2) 
We have
$$g_ig_{i+1}g_i = g_ig_{i+1} (T_i+B'_{i,i+1})=g_ig_{i+1}T_i+B'_{i+1,i+2}g_ig_{i+1}$$
and 
$$g_{i+1}g_ig_{i+1}=(T_{i+1}+B'_{i+1,i+2})g_ig_{i+1}=T_{i+1}g_ig_{i+1}+B'_{i+1,i+2}g_ig_{i+1}.$$ 
Thus it is enough to show that
\eq \label{eq:braid}
g_ig_{i+1}T_i=T_{i+1}g_ig_{i+1}.
\eneq
By \eqref{rel:quad} and \eqref{TkBij}, we get
\eq \label{eq:ggT}
g_ig_{i+1}T_i=T_iT_{i+1}T_i+B'_{i,i+1}T_{i+1}T_i+\left((q-q\inv)B'_{i,i+2}+B'_{i,i+1}  B'_{i,i+2}  \right) T_i+  B'_{i,i+2} .
\eneq 
\smallskip
On the other hand we have
\eqn
T_{i+1}g_ig_{i+1} &&= T_{i+1}(T_i+B'_{i,i+1})(T_{i+1}+B'_{i+1,i+2})\\
&&=T_{i+1}T_iT_{i+1} + T_{i+1}B'_{i,i+1}T_{i+1} + T_{i+1}T_iB'_{i+1,i+2}+ T_{i+1}B'_{i,i+1}B'_{i+1,i+2}.
\eneqn
For the second term, we get
\eqn
T_{i+1}B'_{i,i+1}T_{i+1} = \left((q-q^{-1}) B'_{i,i+2} + (B'_{i,i+2}-B'_{i,i+1})B'_{i+1,i+2}  \right) T_{i+1} +B_{i,i+2}.
\eneqn
For the last two terms, we have
\eqn
&&T_{i+1}T_i B'_{i+1,i+2}+T_{i+1}B'_{i,i+1}B'_{i+1,i+2}\\
&&=
\left(T_{i+1} B'_{i,i+2}T_{i}+ T_{i+1}B'_{i,i+2}B'_{i,i+1}-T_{i+1}B'_{i+1,i+2}B'_{i,i+1} \right)+T_{i+1}B'_{i,i+1}B'_{i+1,i+2}\\
&&=T_{i+1} B'_{i,i+2}T_{i}+ T_{i+1}B'_{i,i+2}B'_{i,i+1}\\
&&=\left( B'_{i,i+1}T_{i+1} T_i + ((B'_{i,i+1}-B'_{i,i+2})B'_{i+1,i+2})T_i \right)+ T_{i+1}B'_{i,i+2}B'_{i,i+1}.
\eneqn
Since $s_{i+1}.(B'_{i,i+2}B'_{i,i+1})=B'_{i,i+2}B'_{i,i+1}$,  by Lemma \ref{Ti_symm_pol} we have
\eqn
T_{i+1}B'_{i,i+2}B'_{i,i+1}=B'_{i,i+2}B'_{i,i+1} T_{i+1}.
\eneqn
Summing up, we obtain
\eq \label{eq:Tgg}
T_{i+1}g_ig_{i+1}=  T_{i+1}T_iT_{i+1}+B'_{i,i+1}T_{i+1}T_i+A \, T_{i+1}+ C \, T_i+B'_{i,i+2},
\eneq
where
\eqn
&&A=(q-q\inv)B'_{i,i+2}-B'_{i,i+1}B'_{i+1,i+2}+B'_{i,i+2}B'_{i+1,i+2}+B'_{i,i+2}B'_{i,i+1} 
\quad \text{and},\\
&& C= (B'_{i,i+1} -B'_{i,i+2} )B'_{i+1,i+2}. 
\eneqn
Note that if $A=0$, then $C=(q-q\inv)B'_{i,i+2}+B'_{i,i+1}  B'_{i,i+2} $ and hence 
by comparing \eqref{eq:ggT} and \eqref{eq:Tgg}, 
we obtain $g_ig_{i+1}T_i=T_{i+1}g_ig_{i+1}$.

Since
\eqn 
&&(q-q\inv)B'_{i,i+2}+B'_{i,i+2}B'_{i+1,i+2}+B'_{i,i+2}B'_{i,i+1} \\
&&=(q-q\inv)^2(-\sum_{k_i<k_{i+2}}b_{k_1\cdots k_n}+\sum_{k_i<k_{i+2}, \, k_{i+1}<k_{i+2}}b_{k_1\cdots k_n}+\sum_{k_i<k_{i+2}, \, k_{i}<k_{i+1}}b_{k_1\cdots k_n})\\
&&=-(q-q\inv)^2\sum_{k_i<k_{i+2}, \, k_i<k_{i+1}, \, k_{i+1}<k_{i+2}}b_{k_1\cdots k_n} \\
&&=-(q-q\inv)^2 \sum_{k_i<k_{i+1}<k_{i+2}}b_{k_1\cdots k_n} 
=-B'_{i,i+1}B'_{i+1,i+2},
\eneqn
we obtain $A=0$,
as desired.
\QED

Hence for each $w\in \sym(n)$, the element
\eq
g_w:=g_{i_1} \cdots g_{i_l}
\eneq
is well-defined,   i.e. independent of the choice of the reduced expression of $w$, where $s_{i_1}\cdots s_{i_l} $ is a reduced expression of $w$.

The lemma below follows immediately from Proposition \ref{prop:quadratic}.
\begin{lemma}  \label{lem:gwgi}
For $w \in \sym(n)$, we have
$$g_wg_{s_i} =\begin{cases} g_{ws_i}\,\,\, \text{if}\,\,\,\ell(ws_i)> \ell(w) \\
	g_{ws_i}+(q-q\inv)g_w  e_i \,\,\, \text{if}\,\,\,\ell(ws_i)< \ell(w).
\end{cases}$$ and
$$g_{s_i}g_w =\begin{cases} g_{s_iw}\,\,\, \text{if}\,\,\,\ell(s_iw)> \ell(w) \\
	g_{s_iw}+(q-q\inv) e_i  g_w\,\,\, \text{if}\,\,\,\ell(s_iw)< \ell(w).
\end{cases}$$ 
\end{lemma}

\begin{lemma}\label{gw}
For $w \in \sym(n)$, we have
$$  g_w \in T_w + \sum_{w'<w}  R [t_1,\ldots, t_n]T_{w'},$$
where 
$<$ denotes the Bruhat order on $\sym(n)$.
\end{lemma}

\Proof We prove by induction  on the length of $w$. 
Suppose that the claim holds for $w\in \sym(n)$ and  
assume that $\ell(s_iw) > \ell(w)$.
We will verify the claim for $s_iw$.
By the assumption, there is $A\in  \sum_{w'<w} R_1[t_1,\ldots, t_n]T_{w'}$ such that
$$g_{s_iw}=g_ig_w=g_i(T_w+A)=(T_i+B'_{i,i+1})(T_w+A)=T_{s_iw}+
B'_{i,i+1}T_w+B'_{i,i+1}A+T_iA.$$

Since $w< s_iw$, it is enough to show that 
$$T_iA \in \sum_{v<s_i w} R[t_1,\ldots, t_n]T_{v}.$$
Indeed we have
\eqn
T_i A &&\in \sum_{w' < w} T_i R[t_1,\ldots, t_n] T_{w'} \\
&&\subset \sum_{w' < w} R[t_1,\ldots, t_n] T_iT_{w'} + \sum_{w' < w}  R[t_1,\ldots, t_n] T_{w'} \\
&&\subset  
\sum_{w' < w} R[t_1,\ldots, t_n] T_{s_iw'} + \sum_{w' < w}  R[t_1,\ldots, t_n] T_{w'}
\eneqn
by Lemma \ref{relationTiwithBasis}  and Lemma \ref{lem:gwgi}.
Notice that if $\ell(s_iw) > \ell(w)$, then $w'<w$ implies $s_iw' < s_iw$, which completes the proof.
\QED

By Theorem \ref{thm:Tbasis} and Lemma \ref{gw}, we get
\Th \label{basisShojiHecke}
The algebra $\mathcal H_{n,r}$  has a $R$-basis given as follows:
$$B:= \set{t_1^{c_1}\cdots t_n^{c_n}g_w}{w \in \sym(n), \  0\le c_i \le r-1}.$$
\end{theorem}

\begin{corollary}
The set
\eq \label{eq:bkbasis}
\B:=\set{b_{k_1,\ldots,k_n}g_w}{(k_1,\ldots,k_n) \in [1,r]^n, \ w\in \sym(n)}
\eneq 
forms an $R$-basis of $\Hnr$.
\end{corollary}

The following is  one of the main theorems of this paper.

\Th \label{thm:main}
The $R$-algebra $\mathcal H_{n,r}$  has the following presentation.

Generators: $t_1,\ldots, t_n$ and $g_1,\ldots,g_{n-1}$

Relations:
\be\label{eq:presentationproperties}
\item $(t_i-u_1)\cdots(t_i-u_r)=0$
\item $t_it_j = t_jt_i$
\item $g_jt_i=t_{s_j(i)}g_j$
\item $g_ig_j=g_jg_i$ if $|i-j| >1$
\item $g_ig_{i+1}g_i = g_{i+1}g_ig_{i+1}$
\item $g_i^2 = 1+(q-q\inv)e_ig_i$,
\ee
where $e_i:=\sum_{k_i=k_{i+1}} b_{k_1,\ldots,k_n}$ and $b_{k_1,\ldots,k_n}:=\displaystyle\prod_{1\le i\le n} \prod_{1\le j \le r, \, j \neq k_i} \dfrac{t_i-u_j}{u_{k_i}-u_j}$.
\end{theorem}
\Proof

Since $g_i=T_i+B'_{i,i+1}$, the set  $\{t_1,\ldots,t_n, g_1,\ldots,g_{n-1} \}$ generates $\Hnr$.
The relations  (3)-(6) were shown through this section.  

Now let $\Hnr'$ be the $R$-algebra defined by  the presentation above. 
Then there is a surjective $R$-algebra homomorphism $\psi:\Hnr' \to \Hnr$,
which maps the generators of $\Hnr'$ to the elements of $\Hnr$ represented by the same symbols.
Let $B'$ be the subset of $\Hnr'$ analogous to $B$ in Theorem \ref{basisShojiHecke}.
By the relations, any elements in $\Hnr'$ can be written as an $R$-linear combination of elements in $B'$.
Moreover $\psi$ maps  $B'$ to $B$, which is linearly independent over $R$. 
Hence  $B'$ is $R$-basis of $\Hnr'$ and $\psi$ is an isomorphism, as desired. 
\QED

Note that $e_i$ commutes with $g_i$.
Hence there exists an $R$-algebra anti-involution 
of $\Hnr$ sending $g_j \to g_j$ $(1\le j\le n)$ and  $t_i\to t_i$  $(1 \le i \le n)$.


\begin{remark} \label{rem:YH presentation}
If we take $R=\C[q, q\inv]$ and $u_k:=e^{\frac{2\pi \sqrt{-1} k}{r}}$ $(1\le k \le r)$, then the algebra \emph{defined by} the above presentation is called the
Yokonuma-Hecke algebra of type $A$ (see \cite{Juyumaya},  \cite{CLPdA}). 
Note that under this choice of base ring and parameters,   we have
\eqn
e_i=\dfrac{1}{r} \sum_{s=0}^{r-1} t_i^s t_i^{r-s}.
\eneqn
In  \cite[Theorem 13]{ERH}, it is shown that
the algebra $\mathcal H_{n,r}(\C[q, q\inv], q, e^{\frac{2\pi \sqrt{-1} k}{r}} \, (1\le k \le r))$ is isomorphic to the Yokonuma-Hecke algebra, which can be understood as a special case of  Theorem \ref{thm:main}.
Note that in this case one can obtain an isomorphism between these two algebras by comparing the results of Lusztig in \cite{Lusztig} and Jacon-Poulain d'Andecy in \cite{JLPdA} on the structure of Yokonuma-Hecke algebras and the ones of Sawada-Shoji in \cite{SawadaShoji} and Hu-Stoll in \cite{HuStoll} on $\Hnr$,  as explained in \cite[Section 3.1]{ERH}.
\end{remark}

\begin{remark}
It is known that (see, \cite[(8.3.2)]{SawadaShoji}) if 
\eq \label{eq:separation}
\text{(separation condition) } \qquad q^{2k} u_i-u_j \in R^\times  \quad \text{for } \  -n < k < n, \, i\neq j,
\eneq
then the algebra $\Hnr$ is isomorphic to the Ariki-Koike algebra associated with the same parameter
(see    \cite[(8.3.2)]{SawadaShoji},  and   see \cite[Section 1.1]{SawadaShoji} for the definition of Ariki-Koike algebras over $R$ associated with $(q,u_1,\ldots,u_r)$).
Hence the theorem above also provides a new presentation of the Ariki-Koike algebra with the   condition \eqref{eq:separation}   on the parameters,  which includes the generic Ariki-Koike algebra. 

\end{remark}

\begin{theorem} \label{thm:presentation2}
There is  a presentation of $\Hnr$ 
with the generators 
$$g_1,\ldots, g_{n-1} \qquad \text{and} \quad b_{k_1,\ldots,k_n} \quad \text{for} \ (k_1,\ldots,k_n) \in [1,r]^n$$
subject to the relations
\begin{enumerate}
\item $g_ig_j=g_jg_i$ if $|i-j| >1$,
\item $g_ig_{i+1}g_i = g_{i+1}g_ig_{i+1}$,
\item $g_i^2 = 1+(q-q\inv)e_ig_i$,
\item $b_{k_1,\ldots,k_n} b_{k'_1,\ldots,k'_n}  = \delta_{(k_1,\ldots,k_n), (k'_1,\ldots,k'_n)}  b_{k_1,\ldots,k_n}$,
\item $g_i b_{k_1,\ldots,k_n} = b_{s_i.(k_1,\ldots, k_n)}  g_i$
\item $\sum b_{k_1,\ldots,k_n}=1$,
\end{enumerate}
where $e_i:=\sum_{k_i=k_{i+1}} b_{k_1,\ldots,k_n}$.
\end{theorem}
\begin{proof}
Let $\Hnr'$ be the $R$-algebra defined by the above presentation. Then there is a surjective $R$-algebra homomorphism $\phi:\Hnr'\to \Hnr$ 
assigning $g_i \mapsto g_i$ and $b_{k_1,\ldots,k_n} \mapsto b_{k_1,\ldots,k_n} $.
Since 
\eq \label{eq:ti}
t_i=\sum u_{k_i} b_{k_1,\ldots,k_n} \qquad (1\le i\le n)\quad \text{in} \ \Hnr,
\eneq
the homomorphism $\phi$ is surjective.

Let  $\B'$ be the subset of $\Hnr'$ analogous to $\B$,  the basis  in  \eqref{eq:bkbasis}.
Then by the relations (1)--(6), any element in $\Hnr'$ can be written as an $R$-linear combination of $\B'$. 
In particular,  the relation (3) together with (6) implies that $g_i^2$ can be written as an $R$-linear   combination of $\B'$.  
Moreover $\phi$ maps $\B'$ to $\B$, which is linearly independent over $R$. Hence $\B'$ is a basis of $\Hnr'$ and $\phi$ is an isomorphism, as desired.
\end{proof}

Note that in the presentation of the theorem above,  the parameters $u_1,\ldots,u_r$ do not show up. Hence we have
\begin{corollary} 
\label{cor:para_indep}
Let $R$ be an integral domain and $q\in R^\times$.
Assume that 
$(u_1,\ldots,u_r )$ and $(\tilde u_1,\ldots,\tilde  u_r)$ be $r$-tuples of elements in $R$ such that
\eq
\Delta=\prod_{i>j} (u_i-u_j) \in  R^\times \quad \text{and} \quad  \tilde \Delta:=\prod_{i>j} (\tilde  u_i-\tilde  u_j)  \in  R^\times.
\eneq
Then 
there is an  isomorphism
\eqn 
\Hnr(R,q,\tilde  u_1,\ldots,\tilde  u_r)  \isoto \Hnr(R,q,u_1,\ldots,u_r  )
\eneqn
of $R$-algebras which assigns 
\eqn
\tilde T_j \mapsto T_j,   \quad \text{and} \quad  \tilde t_i \mapsto \sum_{(k_1,\ldots,k_n)\in [1,r]^n} \tilde  u_{k_i} b_{k_1,\ldots, k_n}
=\sum_{j=0}^{r-1} a_j t_i^j,
\eneqn
where $\tilde T_j$, $\tilde t_i$ denote the generators of $\Hnr(R,q,\tilde  u_1,\ldots,\tilde  u_r)$ of the presentation  in 
Definition \ref{def:Hnr}, and $a_0,\ldots, a_{r-1} \in R$ are given by the equation
\eqn
\begin{pmatrix}
1 & u_1 & u_1^2 & \cdots & u_1^{r-1} \\ 
1 & u_2 & u_2^2 & \cdots & u_2^{r-1} \\
\vdots &\vdots  & \vdots & \ddots & \vdots \\
1 & u_r & u_r^2 & \cdots & u_r^{r-1}
\end{pmatrix}
\begin{pmatrix}
a_0 \\
a_1 \\
\vdots \\
a_{r-1} 
\end{pmatrix}
=
\begin{pmatrix}
\tilde  u_1 \\
\tilde  u_2 \\
\vdots \\
\tilde  u_{r}
\end{pmatrix}.
\eneqn
\end{corollary}
We remark here that the above matrix is the transpose of the Vandermonde matrix $V(u_1,\ldots,u_r)$ in \eqref{eq:Vandermonde}.
\begin{proof}
By Theorem \ref{thm:presentation2},  there is an $R$-algebra isomorphism from
$\Hnr(R,q,\tilde  u_1,\ldots,\tilde  u_r)$ to $\Hnr(R,q,u_1,\ldots,u_r)$ which matches the generators in the presentation in Theorem \ref{thm:presentation2}.   By \eqref{eq:giTi} and \eqref{eq:ti},  the image of $\tilde t_i$ is the same with $\sum \tilde  u_{k_i} b_{k_1,\ldots, k_n}$. 
Since
\eqn
\left(\sum_{k=0}^{r-1} a_j t_i^j \right) b_{k_1,\ldots, k_n} = \left(\sum_{j=0}^{r-1} a_j u_{k_i}^j   \right) b_{k_1,\ldots, k_n} = \tilde  u_{k_i} b_{k_1,\ldots, k_n} \quad \text{for any} \quad (k_1,\ldots k_n) \in [1,r]^n,
\eneqn
we have 
$\sum \tilde  u_{k_i} b_{k_1,\ldots, k_n}
=\sum_{j=0}^{r-1} a_j t_i^j$, as desired.
\end{proof}

\begin{remark}
 The independence on the parameters $u_1,\ldots,u_r$ of the modified Ariki-Koike algebra follows from the structure theorem of $\Hnr$ in \cite{SawadaShoji, HuStoll}. See, in particular, the discussion at the end  of \cite{HuStoll} for the general choice of the base ring $R$.  We emphasize that our isomorphism is more explicit and it is done directly by using only the terms of the standard generators of $\Hnr$.
\end{remark}

\section{Symmetrizing trace form}
Define an $R$-linear map $\tau : \mathcal H_{n,r} \longrightarrow R$ by
\eq
\tau(t_1^{c_1}\cdots t_n^{c_n}g_w)=\begin{cases}
1 & \text{if}\ \  w=\id_{\sym(n)}, c_1=\cdots=c_n=0
\\
0 & \text{ otherwise }
\end{cases}
\eneq
for $0\le c_1,\ldots, c_r \le r-1$, $w \in \sym(n)$.
Note that for any polynomial $f(X)\in R[X]$ in one variable and any $1\le i\le n$, we have
\eq \label{eq:indep_i}
\tau(f(t_i))=\tau(\tilde f(t_i))=a,
\eneq
where $\tilde f(X)$ is the polynomial with degree $<r $ congruent to $f(X)$ modulo $(X-u_1)\cdots (X-u_n)$, and    $a$ is the constant term of $\tilde f(X)$.
Note also that if $f_1(X),\ldots,f_n(X) \in R[X]$, then 
\eq \label{eq:product}
\tau(f_1(t_1)\cdots f_n(t_n)) =  \tau(\tilde f_1(t_1)\cdots \tilde f_n(t_n)) =a_1\cdots a_n 
=\tau(f_1(t_1))\cdots \tau( f_n(t_n)),
\eneq
where $a_i$ denotes the constant term of $\tilde f_i(X)$.

The proof of the below lemma is identical to that of \cite[Proposition 10]{CLPdA}), but we include it here for reader's convenience.
\Lemma {\rm (\cite[Proposition 10]{CLPdA})}
Let $w,w' \in \sym(n)$.
Then
\[
\tau(g_w g_{w'}) = \delta_{w',w^{-1}}.
\]
It is equivalent to saying that
for any $A \in R[t_1,\ldots,t_n]$,
\[ \tau(A g_w g_{w'}) =
\begin{cases}
\tau(A) & \text{if} \ w^{-1}= w' \\
0 & \text{if} \ w^{-1}\neq w'
\end{cases}
\]

\enlemma
\begin{proof}
Since $\{t_1^{c_1}\cdots t_n^{c_n} g_w \}$ is an $R$-basis of $\Hnr$, the second statement follows from the first, by the definition of $\tau$.

When $w=\id$, there is nothing to prove.
Assume that $\ell(w) \ge 1$. Then there is $1\le i \le n-1$ such that $\ell(ws_i) < \ell(w)$.
Then
\[
g_w= g_{(ws_i)s_i } = g_{ws_i} g_{s_i}.
\]

(Case 1)  $\ell(s_iw') > \ell(w')$: We have
\[
g_{s_i}g_{w'}= g_{s_iw'}
\]
On the other hand,  $\ell(ws_i) < \ell(w)$ implies that $\ell(s_iw^{-1}) < \ell(w^{-1})$ so that $w'\neq w^{-1}$ and hence $s_iw^{-1}=(ws_i)^{-1} \neq s_iw'$.
By induction on $\ell(w)$, we have
\[
\tau(g_w g_{w'})=\tau(g_{ws_i} g_{s_i} g_{w'})=\tau(g_{ws_i} g_{s_iw'})=0.
\]
(Case 2)
$\ell(s_iw') < \ell(w')$: We have
\[
g_{s_i}g_{w'}= g_{s_iw'}+(q-q^{-1}) e_ig_{w'}
\]
and hence
\[
\tau(g_wg_w') = \tau( g_{ws_i} g_{s_iw'}+(q-q^{-1}) g_{ws_i}  e_ig_{w'})
=\tau( g_{ws_i} g_{s_iw'}) + (q-q^{-1})\tau( g_{ws_i} e_ig_{w'})
\]
On the other hand, $\ell(s_iw') < \ell(w')$ implies that $\ell((w')^{-1}s_i) < \ell((w')^{-1})$ so that
$(w')^{-1} \neq w {s_i}$.
Since $g_{ws_i} e_i = A g_{ws_i}$ for some $A \in R_1[t_1,\ldots, t_n]$, we have
\[
\tau(g_{ws_i} e_ig_{w'}) = \tau(A g_{ws_i} g_{w'}) =0
\]
by the induction hypothesis.

If $w^{-1}\neq w'$, then $(ws_i)^{-1} \neq s_iw'$, and by induction,
\[
\tau(g_{ws_i}g_{s_iw'})=0
\]

If $w^{-1}= w'$, then $(ws_i)^{-1} = s_iw'$, and by induction,
\[
\tau(g_{ws_i}g_{s_iw'})=1,
\]
as desired.
\end{proof}

Recall that an $R$-linear map $f: \mathcal A \to R$ from an $R$-algebra $\mathcal A$ to its base ring $R$ is called a \emph{trace form} if
\eqn
f(xy)=f(yx) \quad \text{for} \ x,y \in \mathcal A.
\eneqn
A trace form $f$ is called \emph{symmetrizing} if the bilinear form 
\eqn
\mathcal A \times \mathcal A \to R \qquad  \text{given by} \quad (x,y) \mapsto f(xy) 
\eneqn
is non-degenerate.  

\Cor
The map $\tau$ is a trace form on $\Hnr$. 
\encor
\begin{proof} 
Because
$\set{t_1^{c_1}\cdots t_n^{c_n}g_w}{0\le c_i\le r-1, \, w \in  \sym(n)}$ and $\set{g_w t_1^{c_1}\cdots t_n^{c_n}}{0\le c_i\le r-1, \, w \in  \sym(n)}$ are $R$-basis of $\Hnr$,
it is enough to show that
\[
\tau(t_1^{c_1}\cdots t_n^{c_n}g_wg_{w'}t_1^{d_1}\cdots \ t_n^{d_n})
=\tau(g_{w'}t_1^{d_1}\cdots \ t_n^{d_n}t_1^{c_1}\cdots t_n^{c_n}g_w)
\]
for all $w,w' \in \sym(n)$ and $0\le c_i,d_i\le r-1$.
We have
\[
\tau(t_1^{c_1}\cdots t_n^{c_n}g_wg_{w'}t_1^{d_1}\cdots \ t_n^{d_n})
= \delta_{w',w\inv} \tau(t_1^{c_1+d_1}\cdots t_n^{c_n+d_n})
\]
and
\[
\tau(g_{w'}t_1^{d_1}\cdots \ t_n^{d_n}t_1^{c_1}\cdots t_n^{c_n}g_w)
=\tau(t_{w'(1)}^{c_1+d_1}\cdots \ t_{w'(n)}^{c_n+d_n}g_{w'}g_w)
= \delta_{w',w\inv} \tau(t_{w'(1)}^{d_1+c_1}\cdots \ t_{w'(n)}^{c_n+d_n})
\]
Hence it amounts to show that
\[\tau(t_1^{p_1}\cdots t_n^{p_n})=\tau(t_{v(1)}^{p_1}\cdots t_{v(n)}^{p_n})
\]
for any $0\le p_i \le 2r-2$ and  any $v\in \sym(n)$.
Indeed, it follows from
\eqref{eq:indep_i} and \eqref{eq:product}.
\end{proof}

\begin{remark}
The trace form $\tau$ in the case of the Yokonuma-Hecke algebra appeared in \cite{CLPdA}.  It is known to be the same with the trace form obtained from an isomorphism between Yokonuma-Hecke algebra and a matrix algebra over Hecke algebras of symmetric groups  (\cite{JLPdA}).
\end{remark}

\begin{lemma}  \label{lem:tau_on_ti}
For $0\le s \le r-1$ and $\le i\le n$,
we have
\eq
\tau(t_i^{r+s})=\tau(t_i^r) h_s = (-1)^{r+1} (u_1\cdots u_r) h_s = (-1)^{r+1} \sigma_r h_s
\eneq
where $h_s=h_s(u_1,\ldots,u_r)$ denotes the $s$-th complete homogeneous symmetric polynomial in $u_1,\ldots, u_r$.
\end{lemma}

\begin{proof}

Note that for any polynomial $f(X)$ in single variable,  the value $\tau(f(t_i))$ is independent to $i$.
Hence in the proof we will denote $t$ as a representative of $t_i$'s.
Note that $\tau(t^r)= (-1)^{r+1} (u_1\cdots u_r) $ by the defining relation.

We will proceed by induction on $s$. When $s=0$, it is trivial.

Assume that $s\ge 1$.
From  \eqref{eq:t^r}, we have
\eqn
t^{r+s} = \sum_{k=0}^{r-1} (-1)^{r-k+1} \sigma_{r-k} \, t^{k+s},
\eneqn
where $\sigma_{k}=\sigma_{k}(u_1,\ldots,u_r)  $  is the $k$-th elementary symmetric polynomial in $u_1,\ldots, u_r$.
It follows that
\eqn
\tau(t^{r+s})&& = \sum_{k=0}^{r-1} (-1)^{r-k+1} \sigma_{r-k}\,  \tau(t^{k+s})
=\sum_{k=r-s}^{r-1} (-1)^{r-k+1} \sigma_{r-k}\,  \tau(t^{k+s}) \\
&&=\sum_{k=r-s}^{r-1} (-1)^{r-k+1} \sigma_{r-k}\,  \tau(t^r) h_{k+s-r}
=\tau(t^r) \sum_{j=1}^{s} (-1)^{j+1} \sigma_{j}\,  h_{s-j} = \tau(t^r) h_s
\eneqn
where the third equality comes from the induction hypothesis, and the last equality follows from the  well-known Newton's identity below.
\eq
\sum_{j=0}^{s} (-1)^{j} \sigma_{j}\,  h_{s-j}  = \begin{cases} 
1 & \text{if} \ s=0 \\
0 & \text{if} \  s>0.
\end{cases}
\eneq
\end{proof}

\begin{lemma} \label{lem: tdtccheck} 
For $0 \le c, d\le r-1$ and $1\le i \le n$, we have
\eqn
&&\tau\left(t_i^c  \left(\sum_{j=0}^{r-d-1}  (-1)^{j} \sigma_j t_i^{r-d-j}\right)\right) =
 \delta_{c,d} (-1)^{r+1}\sigma_r.
\eneqn
\end{lemma}
\begin{proof}
We have
\eqn
&&\tau\left(t_i^c  \left(\sum_{j=0}^{r-d-1} (-1)^j \sigma_j t_i^{r-d-j}\right) 
\right)
= \tau\left(\sum_{j=0}^{r-d-1} (-1)^j \sigma_j t_i^{r+c-d-j} \right) \\
&&= \sum_{j=0}^{r-d-1} (-1)^j \sigma_j \tau(t_i^{r+c-d-j}) 
= \sum_{j=0}^{c-d} (-1)^j \sigma_j \tau(t_i^{r+c-d-j}) \\
&& = (-1)^{r+1}\sigma_r \sum_{j=0}^{c-d} (-1)^j \sigma_j h_{c-d-j} 
= (-1)^{r+1}\sigma_r  \delta_{c,d},
\eneqn
as desired.
\end{proof}

Assume that $\sigma_r=u_1u_2\cdots u_r \in R^\times$.  For each $1\le i\le n$ and $1\le c\le r-1$, set
\eqn
&& (t_i^c)^\vee := \dfrac{(-1)^{r+1}}{\sigma_r}\sum_{0 \le j \le r-c-1} (-1)^j \sigma_j t_i^{r-c-j}.
\eneqn
Then for each $w \in \sym(n)$ and $(c_1,\ldots,c_n)\in [0,r-1]^n]$, we set 
\eq
&&\left(t_1^{c_1}\cdots t_n^{c_n} g_w\right)^\vee:= g_{w\inv}  \prod_{i : c_i \neq 0}  (t_i^{c_i})^\vee.
\eneq

\begin{prop} \label{prop:dual basis} 
Assume that $u_1,\ldots, u_r \in R^\times$.  
For $(c_1, \ldots,c_n) \in [0,r-1]^n$ and $w\in \sym(n)$, we have 
\eqn
\tau\left( t_1^{c_1}\cdots t_n^{c_n} g_w (t_1^{d_1}\cdots t_n^{d_n} g_{u})^\vee \right) =
\begin{cases} 1 & \text{if $w=u$ and $c_i=d_i$ for all $i$,} \\
0 & \text{otherwise.} \end{cases}
\eneqn
\end{prop}
\begin{proof}
We have
\eqn
&&\tau\left(t_1^{c_1} \cdots t_n^{c_n} g_{w}  \, (t_1^{d_1}\cdots t_n^{d_n} g_{u})^\vee \right) =
\tau\left(t_1^{c_1} \cdots t_n^{c_n} g_{w}  \, g_{u\inv} (t_1^{d_1})^\vee \cdots (t_n^{d_n})^\vee \right) 
\\
&&=\delta_{w,u} \tau\left(t_1^{c_1} \cdots t_n^{c_n} (t_1^{d_1})^\vee \cdots (t_n^{d_n})^\vee \right)  
= \delta(\text{$w=u$ and $c_i=d_i$ for all $i$}).
\eneqn
where the last equality follows from  \eqref{eq:product} and Lemma \ref{lem: tdtccheck}.
\end{proof}

\begin{corollary}
The trace form $\tau$ is symmetrizing if $\sigma_r:=u_1 \cdots u_r$ is invertible in $R$.
\end{corollary}

\section{Symmetric group action on $\Hnr$ and the fixed subalgebra $\Hnr^{\sym(r)}$}
Recall that there is an action of the symmetric group $\sym(r)$ on the subalgebra $R[t_1,\ldots,t_n]$ given by
$$^\sigma b_{k_1,\ldots,k_n} : = b_{\sigma(k_1),\ldots,\sigma(k_n)} \quad \text{for $\sigma$ in $\sym(r)$}. $$
For each $\sigma \in \sym(r)$
we obtain an $R$-linear endomorphism on $\Hnr$, denoted by $\sigma$ again, by setting
$$^\sigma g_w := g_w \quad \text{for $w$ in $\sym(n)$}. $$

The following is a direct consequence of Theorem \ref{thm:presentation2}.
\begin{prop}[cf. [Proposition 2.2 in \cite{JLPdA17}]
The map  $\sigma$ on $\Hnr$ is an $R$-algebra automorphism so that the group $\sym(r)$ acts on $\Hnr$ by $R$-algebra automorphisms.
\end{prop}

For each $1\le i,j \le n$, define 
\eqn
e_{i,j}:=\sum_{(k_1,\ldots,k_n)\in [1,r]^n, \, k_i=k_j } b_{k_1,\ldots,k_n}. 
\eneqn
Then $e_{i,i+1}=e_i$ for $1\le i \le n-1$ and $^\sigma e_{i,j} = e_{i,j}$ for all $\sigma \in \sym(r)$ and $1\le i,j\le n$.

\begin{lemma} \label{lem:eij}
For $1\le i<j \le n-1$, we have
\eqn
g_{j-1} \cdots g_{i+1} e_i g_{i+1}\inv \cdots g\inv_{j-1} 
= e_{i,j}.
\eneqn
\end{lemma}
\begin{proof}
We have
$$g_{i+1} e_{i} =  \sum_{k_i=k_{i+1}} g_{i+1}b_{k_1,\ldots,k_n}  = \sum_{k_i=k_{i+2}}b_{k_1,\ldots,k_n} g_{i+1} = e_{i,i+1} g_{i+1}.$$
By induction on $j-i$, we obtain the assertion.
\end{proof}

Let $[k_1,\ldots,k_n]$ denote the orbit of $(k_1,\ldots,k_n)$ under the action of $\sym(r)$ and define
\eqn
b_{[k_1,\ldots,k_n]} : = \sum_{(\ell_1,\ldots,\ell_n) \in {[k_1,\ldots,k_n]}}  b_{\ell_1,\ldots,\ell_n}
=\sum_{\stackrel{(\ell_1,\ldots,\ell_n) \in[1,r]^n }{\text{such that}  \, \ell_i = \ell_j \, \text{if and only if} \, k_i= k_j} } b_{\ell_1,\ldots,\ell_n}.
\eneqn

\begin{lemma} [cf.  Lemma 4.2 in \cite{JLPdA17}] \label{lem:idem_orbit}
For each $(k_1,\ldots,k_n)\in [1,r]^n$, 
we have
\eqn
b_{[k_1,\ldots,k_n]} = \left(\prod_{ 1\le i<j<\le n,  k_i=k_j} e_{i,j} \right) \left( \prod_{ 1\le i<j<\le n,  k_i\neq k_j} (1-e_{i,j})\right).
\eneqn
\end{lemma}
\begin{proof}
We have
\eqn
\prod_{k_i\neq k_j} (1-e_{i,j}) 
&&=\sum_{(\ell_1,\ldots,\ell_n) \in[1,r]^n} \left(\prod_{k_i\neq k_j} (1-e_{i,j})  \right) b_{\ell_1,\ldots,\ell_n} \\
&&=\sum_{(\ell_1,\ldots,\ell_n) \in[1,r]^n} \left( \prod_{k_i\neq k_j} (1-e_{i,j}) b_{\ell_1,\ldots,\ell_n}   \right) \\
&&=\sum_{(\ell_1,\ldots,\ell_n) \in[1,r]^n} \left( \prod_{k_i\neq k_j} (1- \delta_{\ell_i,\ell_j}) b_{\ell_1,\ldots,\ell_n}   \right) \\
&&=\sum_{\stackrel{(\ell_1,\ldots,\ell_n) \in[1,r]^n }{\text{such that}  \, \ell_i\neq \ell_j \, \text{if} \, k_i\neq k_j} } b_{\ell_1,\ldots,\ell_n}. 
\eneqn

Similarly we obtain 
\eqn 
\prod_{k_i= k_j} e_{i,j} =\sum_{\stackrel{(\ell_1,\ldots,\ell_n) \in[1,r]^n }{\text{such that}  \, \ell_i = \ell_j \, \text{if} \, k_i= k_j} } b_{\ell_1,\ldots,\ell_n},
\eneqn
and hence we have
\eqn
\left(\prod_{k_i=k_j} e_{i,j} \right) \left( \prod_{k_i\neq k_j} (1-e_{i,j})\right)
=\sum_{\stackrel{(\ell_1,\ldots,\ell_n) \in[1,r]^n }{\text{such that}  \, \ell_i = \ell_j \, \text{if and only if} \, k_i= k_j} } b_{\ell_1,\ldots,\ell_n}.
\eneqn
Hence the assertion follows from \eqref{eq:orbits}.
\end{proof}

Let $\Hnr^{\sym(r)}$ be the fixed subalgebra of $\Hnr$ under the action of $\sym(r)$.

\begin{prop}
\hfill
\begin{enumerate}
\item  The set
\eqn
\set{b_{[k_1,\ldots,k_n]}g_w}{[k_1,\ldots,k_n] \in \sym(r) \backslash [1,r]^n  , \ w\in \sym(n)}
\eneqn
forms an $R$-basis of the subalgebra $\Hnr^{\sym(r)}$. 

\item The subalgebra of $\Hnr$ generated by $g_1,\ldots,g_{n-1}, e_1,\ldots, e_{n-1}$ is equal to the subalgebra $\Hnr^{\sym(r)}$.
\end{enumerate}
\end{prop}
\begin{proof}
The first assertion is immediate from the definitions.

By definition the elements $g_1,\ldots,g_{n-1}, e_1,\ldots, e_{n-1}$ belong to $\Hnr^{\sym(r)}$.
The second assertion follows from (1) together with Lemma \ref{lem:eij} and  Lemma  \ref{lem:idem_orbit}.
\end{proof}

\medskip
{\bf Declarations :}\

\smallskip
 
\underline{Ethical Approval} :  
This declaration is not applicable to the content of our submission.

\underline{Funding} : The research of M.\ Kim, the first of this paper, was supported by the National Research Foundation of Korea (NRF) Grant funded by the Korea government(MSIP)
(NRF-2022R1F1A1076214 and NRF-2020R1A5A1016126.

\underline{Availability of data and materials} :
This declaration is not applicable to the content of our submission.



\end{document}